\documentclass[reqno,11pt,a4paper]{amsart}
\usepackage{geometry, empheq}
\usepackage[english]{babel}
\usepackage{amsmath}
\usepackage{amssymb}
\usepackage{amsthm}
\usepackage{graphicx, subcaption}
\usepackage{color}
\usepackage{mhequ}
\usepackage{varioref}
\def\b#1{{\mbox{\boldmath $#1$}}}

\newcommand{\arrow}{\rightarrow}
\newcommand{\ds}{\displaystyle}

\newcommand{\kw}{\rule{2mm}{2mm}}

\newcommand{\bqs}{\begin{equs}}
\newcommand{\eqs}{\end{equs}}
\newcommand{\norm}[1]{{\left\Vert  #1\right\Vert }}
\renewcommand{\l}{\left}
\renewcommand{\r}{\right}
\newcommand{\om}{\Omega}

\newcommand{\myref}[1]{(\ref{#1})}
\newcommand{\reals}{\mathbb{R}}

\renewenvironment{proof}{\b{Proof.}}{\hfill\kw}

\newcounter{theassumption}
\newtheorem{assumption}[theassumption]{\b{Assumption}}

\newtheorem{definition}{\b{Definition}}
\newtheorem{lemma}{\b{Lemma}}
\newtheorem{remark}{\b{Remark}}
\newtheorem{corollary}{\b{Corollary}}
\newtheorem{theorem}{\b{Theorem}}

\author{Pedro Merino}
\email{pedro.merino@epn.edu.ec} 
\address{Research Center on Mathematical Modeling ModeMat, EPN-Quito \\ Escuela Polit\'ecnica Nacional
Ladr\'on de Guevara E11-253 Quito, Ecuador}

\begin{document}
\title[error estimates for the optimal control of Burgers equation]{Finite element error estimates for an optimal control problem governed by the Burgers equation}

\maketitle
\begin{abstract}
We derive a-priori error estimates for the finite-element approximation of a
distributed optimal control problem governed by the steady
one-dimensional Burgers equation with pointwise box constraints on the
control. Here the approximation of the state and the control is done by using piecewise linear
functions. With this choice, an $L^2$ superlinear order of convergence
for the control is obtained; moreover, under a further assumption on the regularity structure of the optimal control this error estimate can be improved to $h^{3/2}$. The theoretical findings are tested experimentally by means of numerical examples.
\end{abstract}

% 35Q53 Burgers Equation
% 49K20 Optimality conditions
% 49J20 Optimal control of pde
% 80M10 FEM
% 49N05 Linear optimal control problems
% 65N12 Stability and convergence of numerical methods
% 41A25 Rate of convergence, degree of approximation
% 90C34 Semi-infinite programming

\section{Introduction}
We consider the finite element approximation of the following optimal control problem of the steady one-dimensional Burgers equation with pointwise control constraints: 
\begin{subequations}
\begin{empheq}[left={\mathbf{(P)} \empheqlbrace}]{align}
\displaystyle &\min J({y,u})=\frac12\norm{ y-y_d}_{L^{2}(0,1)}^2 +\frac\lambda2 \norm{u}_{L^{2}(0,1)}^2\\
&\text{subject to: } \nonumber\\
&\begin{array}{l} \label{eq:burgers0}
-\nu  y'' +yy' =B u \quad \text{in } (0,1), \\
y(0)=y(1)=0,\\
\alpha \leq u(x) \leq \beta, \quad \text{ a.e. in }(0,1).
\end{array}
\end{empheq}
\end{subequations}
The Burgers equation is a well-known one dimensional model for
turbulence and its control has been studied by several authors c.f. \cite{burns91}, \cite{burns91s},\cite{dlrku2004}. 
%Whereas there exist literature concerning the numerical approximation of the stationary and non stationary  equation, for example... 
Our aim in this paper consists in deriving a-priori error estimates for the optimal control problem in the $L^2$-norm.

Finite element approximations for control constrained control problems
in fluid mechanics have been previously considered in \cite{dlrmevex08}  and \cite{camara07} for piecewise constant controls. In particular, in the last, the authors report an error order of $h^2$ if the control space is not discretized, whereas an order of $h$ is obtained for the piecewise constant discretization. It is natural to expect that these error estimates also holds in the case of the Burgers equation using the theory developed in \cite{camara07}. However, if the control space is discretized by piecewise linear functions, results were only obtained for the semilinear case in \cite{cas06} and in \cite{meyroe03} for the linear--quadratic case. Since the optimal control is Lipschitz continuous, its approximations by piecewise linear functions seems to be a natural choice, which in addition piecewise linear functions have less degrees of freedom than piecewise constant functions. Here we aim to perform this task by combining the arguments in \cite{camara07} and \cite{cas06} to obtain a superlinear error of convergence for the $L^2$--norm estimate of the control. In addition, by considering a stronger assumption on the structure of the optimal control and relying on the one--dimensional setting of our problem, we are able to improve the order of the error to $h^{3/2}$.

The paper is organized as follows: first we briefly comment the properties the optimal control problem and its conditions for optimality, next we refer to the finite element method approximation of the Burgers equation and the corresponding error estimates . Next, we discuss the approximation of the optimal control problem by piecewise linear functions by  establishing a superlinear order of convergence for the optimal control. We finish the theory by showing that the superlinear error of convergence can be improved under certain assumptions on the regularity of the optimal control.  Finally, we discuss some numerical experiments to confront our theoretical findings.

\section{The control problem}
We consider the discretization analysis for the following optimal control problem, governed by Burgers equation:
\begin{subequations}\label{e:ocpb}
\begin{empheq}[left={\mathbf{(P)} \empheqlbrace}]{align}
\displaystyle &\min_{(y,u)\in H_0^1(0,1) \times U_{ad}} J({y,u})=\frac12\norm{ y-y_d}^2 +\frac\lambda2 \norm{u}^2\\
&\text{subject to: } \nonumber\\
&\begin{array}{l} \label{eq:burgers}
-\nu  y'' +yy' =B u \quad \text{in } (0,1), \\
y(0)=y(1)=0.
\end{array}
\end{empheq}
\end{subequations}
Here, $U_{ad}$ is the set of admissible controls  defined  by $U_{ad}=\{u \in L^2(0,1): \alpha \leq u\leq \beta \}$  with constants $\alpha$ and $\beta$ satisfying  $\alpha < \beta$.  $\lambda >0$ is the usual Tychonoff parameter. We shall denote by $\norm{\cdot}$ and by $( \cdot, \cdot)$ the norm and the scalar product in $L^2(0,1)$, respectively. $B(x)=\mathcal{X}_\omega(x)$ is the indicator function defined in an open subinterval $\omega \subset \om:= (0,1)$, whereas $\nu$  denotes the viscosity parameter which is assumed that satisfies \eqref{eq:visc}.
For different spaces, the  open ball centered in $u$ with radius $r>0$  will be denoted by $B(u, r)$  if there is no risk of confusion.
\subsection{The state equation equation}
The steady Burgers equation is given by
\begin{subequations}\label{eq:burgers}
\begin{align}
-\nu  y'' +yy' &=f\quad \text{in } (0,1), \\
y(0)=y(1) &=0.
\end{align}
\end{subequations}
The weak formulation of the homogeneous Dirichlet problem for the Burgers equation is as follows: given $f \in L^2(0,1)$, find $y \in H_0^1(0,1)$ such that
\bqs
a(y,\varphi)+b(y,y,\varphi)=(f,\varphi), \quad \forall \varphi  \in H_0^1(0,1), \label{eq:burg1}
\eqs
where: $a : H_0^1(0,1) \times  H_0^1(0,1) \arrow \reals $ is the continuous, bilinear and symetric form defined by
\[ a(\phi, \varphi )=\nu \int_0^1 \phi ' \varphi ' dx, \]
and $b: (H_0^1(0,1))^3 \arrow \reals $ stands for the continuous trilinear form defined by

\[ b(\phi, \varphi,\psi )= \frac{1}{3} \int_0^1 [(\phi  \varphi) ' \psi +\phi \varphi ' \psi ]  dx. \]
The trilinear $b$ enjoys the  following important properties c.f. \cite{Volkwein97}
\begin{subequations}\label{eq:b}
\begin{align}
| b(\phi, \varphi,\psi )| \leq \norm{\phi}_{H_0^1(0,1)}\norm{\varphi}_{H_0^1(0,1)}\norm{\psi}_{H_0^1(0,1)}, &\quad \forall (\phi, \varphi,\psi ) \in (H_0^1(0,1))^3, \label{eq:b0}\\
b(\phi, \varphi,\psi )+ b(\phi, \psi,\varphi ) =0, &\quad\forall (\phi, \varphi,\psi ) \in (H_0^1(0,1))^3, \label{eq:b1} \\
b(\phi, \varphi,\varphi )=0, &\quad\forall (\phi, \varphi) \in (H_0^1(0,1))^2. \label{eq:b2}
\end{align}
\end{subequations}
It is well known, cf. \cite[Theorem 2.10]{Volkwein97} that if the condition
% which throughout this paper will be assumed to fulfill the condition:
\begin{equation}
\norm{f}<\nu^2. \label{eq:visc}
\end{equation}
 holds, the Burgers equation \eqref{eq:burg1} has a unique solution depending on the right hand side.  Indeed,  for every $f \in L^2(0,1)$, there exists $y \in H_0^1(0,1)$ which satisfies \myref{eq:burg1} and fulfill the relation $\norm{y}_{H_0^1(0,1)} \leq \frac1\nu \norm{f}$. In addition, by taking the nonlinearity to the right hand side and relying on elliptic regularity results, it can be shown that $y$ belongs to $H^{2+m}(0,1)$ for every integer $m\geq 0$, provided that $f \in H^m(0,1)$.  In the following,  we link $f \in L^2(0,1)$ to its associated state $y \in H_0^1(0,1)$ as the solution of  \eqref{eq:burg1} and we will indicate this explicitly by writing $y=y(f)$ to emphasize that the state $y$ is generated by the right-hand side $f$.
The following property will be useful in the forthcoming sections.
\begin{lemma}\label{l:st_strngconv}  Let $(u_k)_{k\in \mathbb{N}}$  be a sequence of  functions  which converges weakly to $\bar u$ in $L^2(0,1)$ satisfying $\norm{u_k} < \nu^2$, then the sequence $(y_k)_{k\in \mathbb{N}}$ of the corresponding  associated states converges strongly to $\bar y$ in $H_0^1(0,1)$.
\end{lemma}
\begin{proof}
The result is a straightforward consequence of the properties of the solutions of the Burgers equation and the compactness of the usual embeddings. % Teorema 4.5 Tesis Acevedo P.
\end{proof}

\subsection{Existence of solution for the optimal control problem}
The arguments for proving existence of an optimal control are standard since $U_{ad}$ is a nonempty, closed and convex set in $L^2(0,1)$.  

In the following, $\mathcal{F}_{ad} $ will denote the set  feasible pairs  for  $\mathbf{(P)}$, that is, those pairs $(y,u) \in H_0^1(0,1) \times  U_{ad} $ such that \eqref{eq:burgers} is satisfied with $f=Bu$.  Note that $\mathcal{F}_{ad} $ is nonempty.
\begin{theorem}
If the inequality \myref{eq:visc} holds then the problem $(P)$ has a solution.
\end{theorem}
\begin{remark}
Despite the strict convexity of the objective functional and uniqueness of the solution of the state equation, uniqueness of the optimal control can not be guaranteed since $\mathcal{F}_{ad}$ is not necessarily convex.
\end{remark}

\section{Optimality conditions}
In this section we shall derive first-order necessary and second-order  sufficient conditions for local solutions of  $\mathbf{(P)}$, both play an important role in the derivation of error estimates. Therefore, we make precise the notion of local minimum.
\begin{definition}
A pair  $(\bar y, \bar u) \in \mathcal{F}_{ad}$ will be referred as local optimal pair for $\mathbf{(P)}$ if there exist  positive reals $\rho_u$ and $\rho_y$ such that
\[ J(y,u) \geq J(\bar y,\bar u), \quad  \forall (y,u) \in \mathcal{F}_{ad} \cap \l (B(\bar y, \rho_y) \times B(\bar u, \rho_u)  \r).\]
\end{definition}
For convenience, we introduce the following operator.
\begin{definition} \label{d:R}  
We define the  operator $R: H_0^1(0,1) \times  U_{ad} \arrow  H^{-1}(0,1) $ by  the relation
\bqs
\langle R(y,u), \varphi \rangle=a(y,\varphi)+b(y,y,\varphi)-(Bu,\varphi), \quad \forall \varphi \in H^{-1}(0,1),
\eqs 
where $\langle \cdot,\cdot\rangle$ denotes the dual pairing between  $H^{-1}(0,1)$ and $H_0^1(0,1)$.
\end{definition}
Note that $\langle R(y,u),\varphi \rangle=0$ indicates that $y$ is the  weak solution of the state equation \myref{eq:burgers} associated to the control $u$.

 In the next lemmas we study the differentiability of the operator $R$.

\begin{lemma} \label{l:diff_R}
Let $(w,h) \in L^2(0,1)\times H_0^1(0,1)$. The operator $R$ given in Definition \ref{d:R} has first and second derivatives given by: 

\begin{subequations} \label{eq:d_R}
\begin{align}
 \mbox{} & R'(y,u) : H_0^1(0,1)\times L^2(0,1) \arrow H^{-1}(0,1),\nonumber \\
\mbox{} & \langle R'(y,u)(w,h),\varphi \rangle =a(w,\varphi)+b(w,y,\varphi)+b(y,w,\varphi)-(Bh,\varphi),  \label{eq:d1_R}\\
\text{and} \nonumber\\
\mbox{} & R''(y,u) : (H_0^1(0,1)\times L^2(0,1))^2 \arrow H^{-1}(0,1),\nonumber \\
\mbox{} & \langle R''(y,u)(w_1,h_1)(w_2,h_2),\varphi \rangle =b(w_1,w_2,\varphi)+b(w_2,w_1,\varphi), \label{eq:d2_R}
\end{align}
\end{subequations}
respectively for any $(w,h)$, $(w_1,h_1)$ and $(w_2,h_2)$  in  $H_0^1(0,1)\times L^2(0,1)$ accordingly, and for all $\varphi \in H_0^1(0,1)$.
\end{lemma}

\begin{proof}
The result follows from the linear properties of  $a$, $b$ and the scalar product in $L^2(0,1)$.
\end{proof}
\subsection{First-order necessary conditions}
The following first-order necessary conditions are derived in the spirit of \cite{zoku79}.
\begin{lemma}\label{l:regularity}
Let $(\bar y, \bar u) \in \mathcal{F}_{ad}$ be a local optimal pair for $\mathbf{(P)}$, then $(\bar y,\bar u)$ is a regular point for  $\mathbf{(P)}$ in the sense of \cite{zoku79}.
\end{lemma}
\begin{proof}
The regular point condition of $(\bar y, \bar u)$ for the problem $\mathbf{(P)}$  is achieved by noting that for every $f \in H^{-1}(0,1)$ the linear equation
\[ R'(\bar y,\bar u)(w,h)=f \]
has a unique solution $(w,h)$ of the form $\theta (y-\bar y,u-\bar u)$ , with $(y,u) \in H_0^1(0,1)\times U_{ad} $ and $\theta \geq 0$.
\end{proof}
\begin{theorem}\label{t:fonc}
Let $(\bar y, \bar u)$ be a solution of $\mathbf{(P)}$ such that $\norm{B\bar u} <\nu^2$, then there exists an adjoint state $\bar p \in H_0^1(0,1)$ such that the following optimality system is fullfilled:
\begin{subequations} \label{eq:opt_sys}
\begin{align}
&-\nu \bar y'' + \bar y \bar y' =B \bar u,       &\text{in } (0,1), \quad &\text{with} \quad \bar y(0)=\bar y(1)=0, \\
&-\nu \bar p'' - \bar y \bar p' =\bar y -y_d,  &\text{in } (0,1), \quad &\text{with} \quad \bar p(0)=\bar p(1)=0, \label{eq:adj_eq} \\
& (B\bar p + \lambda \bar u, u -\bar u) \geq 0, & \forall u \in U_{ad}  \label{eq:var_ineq}.
\end{align}
\end{subequations}
Moreover, \eqref{eq:var_ineq} can be equivalently expressed in terms of the projection operator:
\begin{equation} \label{eq:uproj}
\bar u(x)= P_{[\alpha, \beta]} \l(  -\frac1\lambda B \bar p (x)\r)
\end{equation}
\end{theorem}
\begin{proof}
This system obtained by using Lemma \ref{l:regularity} and applying the theory in \cite{zoku79} .
\end{proof}
\begin{remark}
It is worth to point out that extra regularity of the optimal quantities can be deduced using standard elliptic regularity results from \cite{giltru98}. Indeed, by taking $\omega=(0,1)$ and since $y\in H^2(0,1)$ and $y_d \in L^2(0,1)$ we have that $\bar p \in H^2(0,1)  \hookrightarrow C^{1,\frac12}([0,1])$.  From the characterization  of $\bar u$ given by \eqref{eq:uproj} and properties of the projection operator  $P_{[\alpha, \beta]}$, it follows that $\bar u \in C^{0,1}([0,1])$. By bootstrapping arguments on the state equation,  we have that $ \bar y $ solves Poisson's equation with a right hand side in $C^{0,\frac12}([0,1])$ therefore, by elliptic regularity results, we conclude that $\bar y \in C^{2,\frac12}([0,1])$. Furthermore,  if $y_d$ is assumed in to be in $C^{2,\frac12}([0,1])$ then $\bar p$ is also in $C^{2,\frac12}([0,1])$. This high regularity of $\bar p$, however, can not be transferred to the optimal control because the projection operator.
\end{remark}
For our forthcoming analysis, we introduce the \emph{Lagrangian} $\mathcal{L}: H_0^1(0,1) \times L^2(0,1) \times H_0^1(0,1) \arrow \reals$ defined by:
\[
\mathcal{L}(y,u,p)=\frac{1}{2} \norm{y-y_d}+ \frac{\lambda}{2} \norm{Bu}-\langle R(y,u) ,p\rangle,
\]
whose corresponding first and second derivatives  (with respect to the first and second variable)  at $(y,u,p) \in H_0^1(0,1)\times L^2(0,1) \times H_0^1(0,1)$, are given by:

\begin{subequations} \label{eq:d_Lagr}
\begin{align}
\mathcal{L}'(y,u, p)(w,h)=&(y-y_d,w)+\lambda ( u, h)-\langle R'(y,u)(w,h), p\rangle,  \label{eq:d1_Lagr} \\
&\text{for all }  (w,h) \in H_0^1(0,1)\times L^2(0,1) , \nonumber\\
\mathcal{L}''(y,u, p)(w_1,h_1)(w_2,h_2)=&(w_1,w_2)+\lambda (h_1, h_2)-\langle R''(y,u)(w_1,h_1)(w_2,h_2), p\rangle,  \label{eq:d2_Lagr} \\
&\text{for all }  (w_i,h_i) \in H_0^1(0,1)\times L^2(0,1) , i=1,2. \nonumber
\end{align}
\end{subequations}
These expressions allow us to write down optimality system \eqref{eq:opt_sys} in terms of the derivatives of $\mathcal L$ in the following usual way:
\begin{subequations} \label{eq:opt_sys_lag}
\begin{align}
\text{\myref{eq:adj_eq} is equivalent to}&  \quad  \frac{\partial\mathcal{L}}{\partial y}(\bar y,\bar u,\bar p)=0, \quad \text{and}  \\
\text{\myref{eq:var_ineq} is equivalent to}& \quad \frac{\partial\mathcal{L}}{\partial u}(\bar y,\bar u,\bar p) (u-\bar u) \geq 0 \quad \forall u \in U_{ad}.
\end{align}
\end{subequations}
\subsection{Second-order sufficient optimality conditions}
The forthcoming analysis of Section \ref{s:errorest}  concerning the approximation of the optimal control problem by the finite element method, requires the formulation of  second-order sufficient optimality conditions. By the nature of the nonlinearity of the Burgers equation, it shall be notice that the two norm-discrepancy does not occur in our formulation. In order to establish second-order sufficient conditions we introduce the \emph{critical cone}.  For $\tau>0$, we define the set
\[
\om_\tau:=\{x\in (0,1): |\bar p(x)+\lambda \bar u(x)| >\tau\}.
\]
The critical cone $C^\tau_{\bar u}$ consists of those directions $v\in L^2(0,1)$, such that 
\begin{subequations} \label{eq:crit_cone}
\begin{align}
v(x)=0 &\ \text{if}\ x \in \om_\tau  \label{eq:crit_cone_a} ,\\
v(x)\geq 0&\ \text{if}\ x \in \om \backslash \om_\tau \text{ and } \bar u(x)=u_a(x), \, \text{and} \label{eq:crit_cone_b}\\
v(x)\leq 0&\ \text{if}\ x \in \om \backslash \om_\tau \text{ and } \bar u(x)=u_b(x).\label{eq:crit_cone_c}
\end{align}
\end{subequations}
The next theorem states second-order sufficient conditions for $\mathbf{(P)}$. For a better presentation we will use the notation  $\mathcal{L}''(\bar y, \bar u, \bar p)[y,u]^2 =\mathcal{L}''(\bar y, \bar u, \bar p)(y,u)(y,u)$.
\begin{theorem} \label{t:ssc}
Let $(\bar u, \bar y) \in \mathcal{F}_{ad} $ be a feasible pair for $(\mathbf{P})$ satisfying first-order necessary conditions formulated in Theorem \ref{t:fonc}, with adjoint state $\bar p \in H_0^1(0,1)$. In addition, suppose there are $\tau >0$ and $\delta >0$,  such that the coercivity property
\begin{equation}\label{eq:ssc}
\delta \norm{h}^2 \leq (w,w)+\lambda ( h, h)-2(ww',\bar p) ,
\end{equation}
is satisfied for all  $h \in C^\tau_{\bar u}$   and all $w \in H_0^1(0,1)$ such that  $R'(\bar y, \bar u)(w,h)=0$, then there exist constants $\sigma>0$ and $\varepsilon >0$ such that 
\begin{equation} \label{eq:quad_grow}
J(\bar y,\bar u)+\sigma \norm{u -\bar u}^2 \leq J(y,u)
\end{equation}
holds for every $u \in U_{ad} \cap \overline{ B(\bar u , \varepsilon)}$  and $y\in H_0^1(0,1)$ obeying  $R(y,u)=0$.
\end{theorem}
\begin{proof} 
We argue by contradiction by adapting the the proof of Theorem 4.1 in \cite{casretro07}. Therefore, we assume the existence of a sequence $(u_k)_{k\in \mathbb{N} }$ in $U_{ad}$ converging to $\bar u$, such that the sequence $(y_k)_{k\in \mathbb{N}}$ of their associated states converge to $\bar y$. Therefore, $(y_k, u_k)$ is an admissible pair which fulfills the relation
\begin{equation}
J(\bar y,\bar u)+\frac1k \norm{u_k -\bar u}^2 > J(y_k,u_k). \label{eq:no_qgc}
\end{equation}

Let us define the sequence of directions $h_{k}:=\ds \frac{u_k-u}{\rho_k}$, and the sequence $w_k:=\ds\frac{y_k-\bar y}{\rho_k}$, with $\rho_k:=\norm{u_k-\bar u}$, for every  $k\in \mathbb{N}$. Clearly $(h_{k})_{k\in \mathbb{N}}$ is bounded, with $\norm{h_{k}}=1$ and so $(w_k)_{k\in \mathbb{N}}$ is also bounded in $H_0^1(0,1)$; this implies the existence of subsequences denoted again by  $(h_{k})_{k\in \mathbb{N}}$  and $(w_k)_{k\in \mathbb{N}}$ respectively, such that $h_{k} \rightharpoonup  h $  in $L^2(0,1)$ and $w_k \rightharpoonup  w $ in $H_0^1(0,1)$.  Moreover, since $h$ belongs to the closed and convex set (and therefore weakly closed) $C^\tau_{\bar u} \cap \overline{B(\bar u, 1)}$, it follows that  $h$ also belongs to $C^\tau_{\bar u}\cap \overline{B(\bar u, 1)}$ .
 Let us check that $(w,h)$ satisfies $R'(\bar y,\bar u)(w,h)=0$. From the definition of $R$ and Lemma \ref{l:diff_R}, the pair $(w_k,h)$ satisfies:
\begin{empheq}{align} \label{eq:lineq_wk}
a(w_k,\varphi)+b(w_k,y_k,\varphi)+b(\bar y, w_k,\varphi)-(B h,\varphi)=0 \quad \forall \varphi \in H_0^1(0,1).
\end{empheq}
%where $z_k $ is consequence of the mean value theorem  and lies between $y_k$ and $\bar y$, therefore $z_k \arrow \bar y$ in $H_0^1(0,1)$.  
Taking the  limit $k \arrow \infty$ in \eqref{eq:lineq_wk}, by the convergence properties of $(h_{k})_{k\in \mathbb{N}}$ and $(w_k)_{k\in \mathbb{N}}$ we see that the pair $(w,h)$ satisfies the linearized equation:
\begin{equation}\label{eq:lineq_w}
R'(\bar y,\bar u)(w,h)=a(w,\varphi)+b(w, \bar y,\varphi)+b(\bar y, w,\varphi)-(B h,\varphi)=0 \quad \forall \varphi \in H_0^1(0,1). 
\end{equation}
% because of Lemma \ref{l:st_strngconv}, the sequence of of their associated states, denoted again by $(y_k)_{k\in \mathbf{N}}$ converges $\bar y$ strongly in $H_0^1(0,1)$.

%The proof is divided in the following three steps: first we show that $h \in C_{\bar u}^{\tau}$, next we conclude %that  $h=0$  and finally we get a contradiction.
%
%
By applying the mean value theorem to the Lagrangian in \eqref{eq:no_qgc}, we obtain 
\begin{equation}\label{eq:Lagr_est1}
\rho_k\mathcal{L}'(\xi_k, \zeta_k,\bar p)(w_k,h)= \mathcal{L}(y_k,u_k,\bar p)-\mathcal{L}(\bar y,\bar u,\bar p)=J(\bar y,\bar u)-J(y_k, u_k)< \frac1k \norm{u_k -\bar u}^2, 
\end{equation}
where $\xi_k$ is between $y_k$ and $\bar y$, and  $\zeta_k$ is between $w_k$ and $w$. Since $h \rightharpoonup h$ in $L^2(0,1)$, Lemma \ref{l:st_strngconv} implies that $w_k \arrow w $ in $L^2(0,1)$ therefore, from \eqref{eq:Lagr_est1} we arrive to 
\begin{equation}
\mathcal{L}'(\bar y,\bar u,\bar p)(w,h) \leq 0.
\end{equation}
On the other hand, we find that $\mathcal{L}'(\bar y, \bar u, \bar p)(w_k,h) \geq 0$ holds for every $k\in \mathbb{N}$  in view of  the first-order necessary conditions expressed in \eqref{eq:opt_sys_lag}. After passing to the limit $k \arrow \infty$ and using the same convergence arguments it follows that
\begin{equation}
\mathcal{L}'(\bar y, \bar u, \bar p)(w,h) \geq 0,
\end{equation}
thus we have that $\mathcal{L}'(\bar y, \bar u, \bar p)(w,h)=0$.

Now, we show that $h=0$.
% by proving that $ (w,w)+\lambda(Bh,Bh)-2(ww',\bar p) \leq \delta \norm{h}^2$ since the last inequality only holds if $h=0$ due to the% second-order sufficient conditions \eqref{eq:ssc}. 
We recall that if $h=0$ then $w=0$ because $w$ is the unique solution of the linearized equation \eqref{eq:lineq_w}. By using the second-order Taylor expansion of the Lagrangian and having in mind \eqref{eq:d2_R}  and \eqref{eq:d2_Lagr} we get
\begin{equation}
\rho_k \mathcal{L}'(\bar y,\bar u,\bar p)(w_k,h)+\frac{\rho_k^2}{2}\mathcal{L}''(\bar y, \bar u,\bar p)[w_k,h]^2=\mathcal{L}(y_k,u_k,\bar p)-\mathcal{L}(\bar y,\bar u,\bar p), \label{eq:Lagr_est2}
\end{equation}
%with $\tilde\xi_k$ is an intermediate point between $\bar y$ and $ y_k$, and $\zeta_k$ is an intermediate point between $\bar u$ and $u_k$. Then we obtain that
hence, by using  \eqref{eq:Lagr_est2}  in \eqref{eq:no_qgc} we estimate
\begin{align}
\rho_k \mathcal{L}'(\bar y,\bar u,\bar p)(w_k,h)+
&\frac{\rho_k^2}{2} ( (w_k,w_k)+\lambda( h_{k},  h_{k})-2(w_kw_k',\bar p))
< \frac{\rho_k^2}{k}.
%\frac{\rho_k^2}{2} \l[\mathcal{L}''(\bar y, \bar u,\bar p)-  %\mathcal{L}''(\tilde\xi_k, \tilde \zeta_k,\bar p) \r][w_k,h]^2
\end{align}
From the definition of $w_k$ and $h_{k}$ and \eqref{eq:opt_sys_lag} the first term in the last inequality is nonnegative, therefore we have
\begin{equation} \label{eq:Lagr_est3}
(w_k,w_k)+\lambda ( h_{k}, h_{k})-2(w_kw_k',\bar p)<\frac{2}{k}.
\end{equation}
Once again, since $h_{k} \rightharpoonup  h $  in $L^2(0,1)$ and $w_k \rightharpoonup  w$ in $H_0^1(0,1)$ then
\begin{equation}
 (w,w)+\lambda ( h,h)-2(w w',\bar p) \leq \liminf_{k\arrow \infty} \, (w_k,w_k)+\lambda ( h_{k}, h_{k})-2(w_kw_k',\bar p)  \leq 0,
\end{equation}
which together with second-order condition \eqref{eq:ssc} implies that $h=0$. Finally, by observing that $\norm{h}=1$,
%  and the convergence of the sequences  $(h)_{k \in \mathbb{N}}$ and $(w_k)_{k \in \mathbb{N}}$, 
 we can infer from \eqref{eq:Lagr_est3}  the final contradiction:
\begin{align}
\lambda &=\liminf_{k \arrow 0}(w_k,w_k)+\lambda-2(w_kw_k',\bar p) \nonumber \\
&  \leq \liminf_{k \arrow 0} (w_k,w_k)+\lambda ( h_{k}, h_{k})-2(w_kw_k',\bar p) \leq 0. \nonumber
\end{align}
\end{proof}
\section{Finite element approximation of the Burgers equation}
This section is devoted to the approximation of  Burgers equation by using the finite element method and the derivation of the corresponding error of convergence.  Let $n$ be a positive integer, we define $h:=1/n$  and  a uniform mesh on the interval $[0,1]$ denoted by  $\mathcal{I}_h$, which consists of $n$ subintervals: $I_i=[x_{i-1},x_{i}]$ of $[0,1]$ for $i=1,\ldots,n$, such that $0=x_0 <x_1 \ldots <x_n=1$ and  $[0,1]=\cup_{i=1}^nI_i$. We also introduce the finite dimensional space $V_h \subset H_0^1(0,1)$ defined by
\begin{align}
V_h&=\{y_h\in C([0,1]): y_{h|I_i} \in \mathcal{P}_1 , \text{ for }i=1,\ldots,n, \, \, \text{with} \,\, y_h(0)=y_h(1)=0\}, \nonumber
\end{align}
where $\mathcal{P}_1$ is the space of polynomials of degree less or equal than one.  Therefore, we define  the discrete Burgers equation in $V_h$  as follows: given $f\in L^2(0,1)$ find $y_h \in V_h$ such that 
\begin{equation}\label{eq:discrete_Burgers}
a(y_h, \varphi_h)+b(y_h,y_h, \varphi_{h})=(f, \varphi_h), \quad \forall \varphi_h \in V_h.
\end{equation}
\begin{theorem}\label{t:y_h} If $f \in U_h$ is such that $\norm{f}<\nu^2$ then, equation \eqref{eq:discrete_Burgers} has a unique solution $y_h \in V_h$ such that
\begin{equation} \label{eq: y_h_bound}
\norm{y_h}_{H_0^1(0,1)} \leq \frac{1}{\nu} \norm{f}.
\end{equation}
\end{theorem}
\begin{proof}
The proof is completely analogous to the proof  in  \cite[Theorem 2.10]{Volkwein97}.
\end{proof}
\\

Let us denote by $\Pi _h : C([0,1]) \arrow V_h$ the usual Lagrange interpolation operator such that for every $z \in H_0^1(0,1)$,
the element $\Pi_h z$ is the unique element in $V_h$ which satisfies $\Pi_h z(x_i)=z(x_i)$ for $i=0,1,\ldots, n$.  

For convenience, we recall a well known result which establishes  an estimate for the interpolation error cf. \cite{cialio91}.
\begin{lemma} \label{l:proj_err}
Let nonnegative integers $m$ and $k$ and $p,q \in [1,\infty]$. If the embeddings 
\begin{empheq}{align}
W^{k+1,p}(T)& \hookrightarrow C^{0}(T), \,\text{ and} \nonumber \\
W^{k+1,p}(T)& \hookrightarrow W^{m,q}(T)  \nonumber
\end{empheq}
hold, then there exists  a constant $C>0$ independent of $h$ such that the following interpolation error is satisfied
\begin{equation} \label{eq:intep_error}
\norm{y-\Pi_T y}_{W^{m,q}(T)} \leq C h^{n(\frac1q-\frac1p)+k+1-m}\norm{y}_{W^{k+1,p}(T)} ,
\end{equation}
where $\Pi_T y$ is the restriction of $\Pi_h y$ to an element $T$ of the discretization of the domain with dimension $n$.
\end{lemma}
Moreover, Lemma \ref{l:proj_err}  implies that
\begin{equation}\label{eq:superlinear1}
\lim_{h\arrow 0} \frac{1}{h} \norm{z-\Pi_h z}=0, \quad \forall z \in W^{1,p}(0,1), \, \text{ and } 1<p. %<+\infty.
\end{equation}
The proof for this result can be found in \cite[Lemma 7]{casmat01}.
We are interested in the error estimate for the approximation of the solution of  the Burgers equation using linear finite elements, to this purpose we convent that $C$ denotes a generic constant which is positive and independent of $h$.
\begin{lemma}  \label{l:proj_err_y}
Let $f \in L^2(0,1)$ be such that $\norm{f}<\nu^2$ and $y \in H_0^1(0,1)$ such that $R(y,f)=0$. If $y_h$ denotes the corresponding solution of the discrete equation \eqref{eq:discrete_Burgers} with right-hand side $f$; then, the estimate
\begin{equation} \label{eq:y_interp}
\norm{y-y_h}_{H_0^1(0,1)} \leq  C \norm{y- \Pi_h y}_{H_0^1(0,1)},
\end{equation}
is satisfied. %, where $C$ is a positive constant independent of $h$.
\end{lemma}
\begin{proof}
Since $y \in H_0^1(0,1)$  and $y_h \in V_h$ satisfy equations \eqref{eq:burg1} and \eqref{eq:discrete_Burgers} respectively, after subtracting both equations we get
\begin{equation} \label{eq:yest_1}
a(y-y_h, \varphi_h)+b(y,y,\varphi_h)-b(y_h,y_h,\varphi_h)=0, \, \forall \varphi_h \in V_h.
\end{equation}
In particular, if $z_h$ is an arbitrary element in $ V_h$, we choose $\varphi_h=y_h-z_h$ in \eqref{eq:yest_1}, resulting in
\begin{align} \label{eq:yest_2}
a(y_h-z_h, y_h-z_h)&=a(y-z_h,y_h-z_h)+b(y,y,y_h-z_h)-b(y_h,y_h,y_h-z_h)
\end{align}
Let us estimate the right-hand side of  \eqref{eq:yest_2}. In view of \eqref{eq:b1}  and  \eqref{eq:b2} we find that
\begin{align}
b(y,y&,y_h-z_h)-b(y_h,y_h,y_h-z_h) \nonumber\\
& =b(y,y-z_h,y_h-z_h)+b(y,z_h,y_h-z_h) -b(y_h,y_h,y_h-z_h)\nonumber \\
& =b(y,y-z_h,y_h-z_h)+b(y-z_h,y_h,y_h-z_h)-b(y_h-z_h,y_h,y_h-z_h)\nonumber
\end{align}
using  \cite[Lemma 3.4, p.9]{Volkwein97} and inequality  \eqref{eq: y_h_bound} we find out that
\begin{align} \label{eq:yest_3}
b(y,&y,y_h-z_h)-b(y_h,y_h,y_h-z_h) \nonumber\\
&\leq (\norm{y}_{H_0^1(0,1)}+\norm{y_h}_{H_0^1(0,1)}) \norm{y-z_h}_{H_0^1(0,1)}\norm{y_h-z_h}_{H_0^1(0,1)}+\norm{y_h}_{H_0^1(0,1)}\norm{y_h-z_h}^2_{H_0^1(0,1)}  \nonumber\\
& \leq \frac2\nu \norm{f} \norm{y-z_h}_{H_0^1(0,1)}\norm{y_h-z_h}_{H_0^1(0,1)} + \frac1\nu \norm{f}\norm{y_h-z_h}^2_{H_0^1(0,1)}\nonumber\\
& \leq 2\nu\norm{y-z_h}_{H_0^1(0,1)}\norm{y_h-z_h}_{H_0^1(0,1)} + \frac1\nu \norm{f}\norm{y_h-z_h}^2_{H_0^1(0,1)}.
%& \leq 2\nu \norm{y-z_h}_{H_0^1(0,1)}\norm{y_h-z_h}_{H_0^1(0,1)} +\nu\norm{y_h-z_h}^2_{H_0^1(0,1)}
\end{align}
Using \eqref{eq:yest_3} in  identity \eqref{eq:yest_2}, the continuity of $a$ and $b$ implies that
\begin{align}
\nu \norm{y_h-z_h}^2_{H_0^1(0,1)}%&=a(y-z_h,y_h-z_h)+b(y,y-z_h,y_h-z_h)+b(y-z_h,z_h,y_h-z_h) \nonumber\\
							& \leq 3\nu\norm{y-z_h}_{H_0^1(0,1)}\norm{y_h-z_h}_{H_0^1(0,1)}+ \frac1\nu \norm{f}\norm{y_h-z_h}^2_{H_0^1(0,1)},\nonumber
\end{align}
from which, we conclude that
\begin{equation}\label{eq:yest_4}
\nu\l(1 -\frac{\norm{f}}{\nu^2} \r)  \norm{y_h-z_h}^2_{H_0^1(0,1)} \leq 3\nu \norm{y-z_h}_{H_0^1(0,1)}\norm{y_h-z_h}_{H_0^1(0,1)} .
\end{equation}
Observe that the coefficient on the left-hand side is a positive number.
Taking $z_h= \Pi_h y$ in  \eqref{eq:yest_4} and using the fact that $\Pi_h$ is a continuous operator, it follows that 
\begin{align}
\norm{y_h-\Pi_h y}_{H_0^1(0,1)}& \leq C  \norm{y-\Pi_h y}_{H_0^1(0,1)}.
\end{align}
Finally,  the last inequality implies the desired estimate as follows
\begin{align}
\norm{y-y_h}_{H_0^1(0,1)}&  \leq \norm{y-\Pi_h y}_{H_0^1(0,1)}+\norm{y_h-\Pi_h y}_{H_0^1(0,1)} \leq C \norm{y-\Pi_h y}_{H_0^1(0,1)} .
\end{align}
\end{proof}

Combining Lemmas \ref{l:proj_err}  and \ref{l:proj_err_y} we arrive to the following result.

\begin{theorem} \label{t:burg_est_h1}
Let $f \in L^2(0,1)$ be such that $\norm{f}<\nu^ 2$ and let $y \in H_0^1(0,1)$ and $y_h \in V_h$ be the solutions of the equations \eqref{eq:burg1} and \eqref{eq:discrete_Burgers} respectively. Then the estimate
\begin{equation}\label{eq:burg_est_h1}
\norm{y-y_h}_{H_0^1(0,1)} \leq C h \norm{y}_{H^2(0,1)}
\end{equation}
is fulfilled. %, where the constant $C$ does not depend on the mesh parameter $h$.
\end{theorem}

In the process of deriving error estimates for the finite element approximation of the optimal control problem $\mathbf{ (P)}$, we will need the following estimate in the $L^2$--norm.

\begin{theorem} \label{t:burg_est_l2}
Let $f \in L^2(0,1)$ be such that $\nu^ 2 > \norm{f}$ and let $y \in H_0^1(0,1)$ and $y_h \in V_h$ the solutions of the state equations \eqref{eq:burg1} and \eqref{eq:discrete_Burgers} respectively. Then, the estimate
\begin{equation}\label{eq:burg_est_L2}
\norm{y-y_h}_{L^2(0,1)} \leq C h^2 \norm{y}_{H^2(0,1)}
\end{equation}
is fulfilled. %, where the constant $C$ does not depend on the mesh parameter $h$.
\end{theorem}
\begin{proof}
In order to derive the $L^2$--estimate for the approximation error of the Burgers equation, we introduce the following auxiliary linear problem:
\begin{empheq}[left={\empheqlbrace}]{align}
&\text{Given} \, r\in L^2(0,1), \,\text{find} \, z \in H_0^1(0,1) \, \text{such that:} \nonumber\\
&a(z,\varphi)+b(y,\varphi,z)+b(\varphi, y, z)=(r,\varphi), \quad \forall \varphi \in  H_0^1(0,1), \label{eq:aux_1}
\end{empheq} 
and its finite element approximation:
\begin{empheq}[left={\empheqlbrace}]{align}
&\text{Given} \, r\in L^2(0,1), \,\text{find} \, z_h \in V_h \, \text{such that:} \nonumber\\
&a( z_h,\varphi_h)+b(y,\varphi_h, z_h)+b(\varphi_h,y,z_h)=(r,\varphi_h), \quad \forall \varphi_h \in V_h.\label{eq:aux_2}
\end{empheq} 

Based on the properties of $b$, it is clear that equations \eqref{eq:aux_1} and \eqref{eq:aux_2}  fulfill the hypothesis of the Lax-Milgram theorem in their respective formulation spaces. Indeed, from its definition the symmetric bilinear form $\tilde a(z,\varphi):=a(z,\varphi)+b(y,\varphi, z)+b(\varphi,y,z)$ is continuous in $H_0^1(0,1)$ from its definition. V-ellipticity follows from \eqref{eq:b} and the estimate $ \norm{y}_{H_0^1(0,1)} \leq \frac1\nu \norm{f}$ as follows:
\begin{align}
\tilde a(\varphi, \varphi)=a(\varphi,\varphi)+b(\varphi, y,\varphi) &\geq (\nu-\norm{y}_{H_0^1(0,1)}) \norm{\varphi}_{H_0^1(0,1)}^2 \nonumber \\
&>\nu(1-\frac{\norm{f}}{\nu^2}) \norm{\varphi}^2_{H_0^1(0,1)}>0.
\end{align}
 Therefore, there exist unique solutions $z \in H_0^1(0,1) \cap H^2(0,1)$ of  \eqref{eq:aux_1} and $z_h \in V_h$ of \eqref{eq:aux_2}, respectively. Moreover, by linearity we can easily check that $z$ and $z_h$  satisfy 
\begin{equation}\label{eq:aux_est}
\norm{z- z_h}_{ H_0^1(0,1)} \leq C h \norm{z}_{H^2(0,1)}. 
\end{equation}

Now, let us observe that  \eqref{eq:yest_1} implies that $y$ and $y_h$ fulfill the relation 
\begin{equation}
a(y-y_h,z_h)=-b(y,y-y_h, z_h)-b(y-y_h,y_h, z_h). \label{eq:aux_5h}
\end{equation}
On the other hand,  taking $\varphi=y-y_h$ in \eqref{eq:aux_1}  we have 
\begin{align}
(r, y-y_h)&=a(z,y-y_h)+b(y-y_h,y,z)+b(y,y-y_h, z)\nonumber \\
	       &=a(z-z_h,y-y_h)+a(z_h,y-y_h)+b(y-y_h,y,z)+b(y,y-y_h, z), \,\nonumber 
\end{align}
where we  replace the identity \eqref{eq:aux_5h} to attain
\begin{align}
(r, y-y_h)=& a(z-z_h,y-y_h)+b(y-y_h,y,z)+b(y,y-y_h, z)\nonumber \\
	       &	-b(y,y-y_h, z_h)-b(y-y_h,y_h, z_h)\nonumber \\
		=& a(z-z_h,y-y_h)+b(y-y_h,y-y_h,z)+b(y-y_h,y_h,z-z_h) \nonumber\\
		&+b(y,y-y_h, z-z_h),\nonumber
\end{align}
then, by continuity of $a$ and $b$ we estimate		
\begin{align}		
(r, y-y_h)	\leq &C \norm{z-z_h}_{H_0^1(0,1)}\norm{y-y_h}_{H_0^1(0,1)}+\norm{z} _{H_0^1(0,1)}\norm{y-y_h}^2_{H_0^1(0,1)}  \nonumber\\
	&+(\norm{y} _{H_0^1(0,1)}+\norm{y_h} _{H_0^1(0,1)}) \norm{y-y_h}_{H_0^1(0,1)} \norm{z-z_h}_{H_0^1(0,1)}, \nonumber
\end{align}
finally, by  using  \eqref{eq:burg_est_h1} and \eqref{eq:aux_est} we arrive to
\begin{equation}
\norm{y-y_h}=\sup_{\norm{r} \leq 1} (r, y-y_h) \leq C h^2 \norm{y}_{H^2(0,1)}, \nonumber
\end{equation}
which finishes the proof.
\end{proof}

\section{Numerical approximation of the control problem}
For convenience, we use the following notation:
\begin{itemize}
\item  For every control $u \in U_{ad} $ satisfying $\norm{Bu} < \nu^2$, $y(u)$ denotes the unique solution   of \eqref{eq:burgers} in $H_0^1(0,1)$.
\item For every control $u \in U_{ad} $ satisfying $\norm{Bu} < \nu^2$, $y_h(u)$ denotes the unique solution of \eqref{eq:discrete_state} in $V_h$.
\end{itemize}
$p(u)$ and $p_h(u)$ will be used analogously to denote the corresponding adjoint states.
\\

Let us define the set of discrete admissible controls by $U_{ad,h}=U_{ad} \cap V_h$. In addition, the  state equation is approximated by the following problem: for a given $u \in U_{ad}$, find $y_h \in V_h$ satisfying
\begin{equation}\label{eq:discrete_state}
a(y_h, \varphi_h)+b(y_h,y_h, \varphi)=(B u, \varphi_h), \quad \forall \varphi_h \in V_h.	
\end{equation}

We are interested in unique local solutions close to the optimal state $\bar y$. With respect to this, we have the following preliminar result.
\begin{lemma} \label{l:local_sol}
Let $(\bar y, \bar u) \in H_0^1(0,1) \times U_{ad}$ the optimal pair for $\mathbf{(P)}$ with $\nu^2 >\norm{B\bar u}$. Then, there exist positive numbers $\rho_1$ and $\rho_2$ independent of the mesh parameter $h$, such that for all $ u \in B(\bar u, \rho_1) $ there exists a unique $y_h(u) \in V_h\cap B(\bar y, \rho_2)$  satisfying equation \eqref{eq:discrete_state}. Moreover, the corresponding discrete state $y_h(u)$ satisfies
\begin{equation}
\norm{y_h(u)}_{H_0^1(0,1)} \leq \frac1\nu \norm{Bu}. \nonumber
\end{equation} 
\end{lemma}
\begin{proof} If we define $\delta:=\nu^2-\norm{B\bar u}$ and $\rho_1:=\frac{\delta}{2}$, then for any $u\in B(\bar u,\rho_1) \cap U_{ad}$ we have that $\norm{B u} \leq \norm{Bu-B\bar u} + \nu^2 - \delta \leq \norm{u-\bar u} + \nu^2 - \delta<\nu^2$. According to Theorem \ref{t:y_h} there exists $y_h=y_h(u)$ satisfying equation   \eqref{eq:discrete_state}, with the bound $\norm{y_h}_{H_0^1(0,1)} \leq \frac1\nu \norm{Bu} $.  $\rho_2$ can be chosen  using the estimate
\begin{align}
\norm{\bar y -y_h(u)}_{H_0^1(\om)}&\leq \norm{\bar y -y(u)}_{H_0^1(,1)}+\norm{ y(u) -y_h(u)}_{H_0^1(0,1)} \nonumber\\
							&\leq \frac1\nu \norm{B(\bar u - u)}+C h \norm{y(u)}_{H^2(0,1)}. \nonumber
\end{align}
\end{proof}
\begin{lemma}  
Let $(\bar y, \bar u)\in H_0^1(0,1) \times U_{ad}$ an optimal pair for $(\mathbf{P})$ with $\nu^2 >\norm{B\bar u}$. Consider controls $u$ and $v$ in the open ball $B(\bar u, \rho_1)$ from Lemma \ref{l:local_sol} , then it follows that
\begin{equation}\label{eq:uv_estimate}
\norm{y(u)-y_h(v)}_{H_0^1(0,1)} \leq C ( h \norm{y(h)}_{H^2(0,1)}+\norm{u-v}) 
\end{equation}
\end{lemma}
\begin{proof}
The proof is analogous to the proof of Lemma \ref{l:proj_err_y}. Since $y_h(u)$ and $y_h(v)$ are the solutions of \eqref{eq:discrete_state} with right-hand side $u$ and $v$ accordingly, we subtract the corresponding equations to satisfy 
\begin{align}\label{eq:eq_uv}
a(y_h(v)-y_h(u),\varphi_h)+b(y_h(v),y_h(v),\varphi_h)-b(y_h(u),y_h(u),\varphi_h)=(B(v-u),\varphi_h),
\end{align}
for all $ \varphi_h \in V_h$. In particular, choosing $\varphi_h=y_h(v)-y_h(u)$ in \eqref{eq:eq_uv} and  using \eqref{eq:b2} we estimate
\begin{align}
\nu \norm{y_h(v)-y_h(u)}^2_{H_0^1(0,1)}=&b(y_h(u),y_h(u),y_h(v)-y_h(u))-b(y_h(v),y_h(v),y_h(v)-y_h(u)) \nonumber\\
&+(B(v-u),y_h(v)-y_h(u)), \nonumber \\
=&b(y_h(u)-y_h(v),y_h(u),y_h(v)-y_h(u)) + (B(v-u),y_h(v)-y_h(u)) \nonumber\\
\leq&\norm{y_h(u)}_{H_0^1(0,1)} \norm{y(v)-y(u)}_{H_0^1(0,1)} ^2+ \norm{v-u}\norm{y(v)-y(u)}_{H_0^1(0,1)}.\nonumber
\end{align}
Taking the first term on the right  to the left side, and taking into account Lemma \ref{l:local_sol} we get
\begin{align}
(\nu -\frac1\nu \norm{B  u} )\norm{y_h(u)-y_h(v)}^2_{H_0^1(0,1)} \leq& \norm{u-v}\norm{y_h(v)-y_h(u)}_{H_0^1(0,1)}, \nonumber
\end{align}
thus, we get the estimate 
\begin{equation}\label{eq:yuyv_estimate}
\ds\norm{y_h(u)-y_h(v)}_{H_0^1(0,1)} \leq \frac{\nu}{\nu^2-\norm{Bu}}  \norm{u-v}.
\end{equation}

 By noticing that $\norm{y(u)-y_h(v)}_{H_0^1(0,1)} \leq \norm{y(u)-y_h(u)}_{H_0^1(0,1)}+ \norm{y_h(u)-y_h(v)}_{H_0^1(0,1)} $,  it is easy to derive \eqref {eq:uv_estimate}  using the estimate \eqref{eq:yuyv_estimate} and the error bound established in Theorem \ref{t:burg_est_h1}.
\end{proof}
% Se puede simplificar si se toma las ecuaciones de y(u) y y(v)
%
\\
\\
We are in place to formulate the discrete optimal control problem associated to $\mathbf{(P)}$. Let us define the discrete admissible set $U_{ad,h}:=U_{ad} \cap V_h$. In addition, we shall not consider any source of error on $y_d$. The discrete optimal control problem is given by:
\begin{empheq}[left={\mathbf{(P_h)} \empheqlbrace}]{align}
\displaystyle &\min_{(y, u)\in V_h \times U_{ad,h}} J({y,u})=\frac12\norm{y-y_d}^2 +\frac\lambda2 \norm{u}^2 \nonumber \\
&\text{ subject to \myref{eq:discrete_state}. }\nonumber
\end{empheq}

\begin{remark} The set $\mathcal F_{ad,h}:=\{(y,u)\in V_h \times U_{ad,h}:  (y,u) \, satisfiying \,  \eqref{eq:discrete_state} \}$ is not empty. Therefore, existence of a solution of $\mathbf{(P_h)}$ is a direct consequence of the compactness of $\mathcal{F}_{ad,h}$ and continuity of $J$ in $\mathcal{F}_{ad,h}$.
 \end{remark}
The optimality system for a local solution $\bar u_h$ of $\mathbf{(P_h)}$ can be derived analogously to the continuous optimality system; therefore, we state this without proof in the following theorem.
\begin{theorem}\label{t:fonch}
Let $(\bar y_h, \bar u_h) \in \mathcal{F}_{ad,h}$ be a local solution of $\mathbf{(P_h)}$ such that $\norm{B\bar u} <\nu^2$, then there exists a discrete adjoint state $\bar p_h \in V_h$ such that the following optimality system is fulfilled:
\begin{subequations} \label{eq:opt_sys_h}
\begin{align}
&a(\bar y_h,\varphi_h) +b(\bar y_h,\bar y_h,\varphi_h) =(B \bar u_h,\varphi_h), & \forall \varphi_h \in V_h, \label{eq:state_h}\\
&a(\bar p_h,\varphi_h) -(\bar y_h \bar p'_h,\varphi_h) =(y_d-\bar y_h,\varphi_h), & \forall \varphi_h \in V_h,  \label{eq:adj_h} \\
 &(B\bar p_h + \lambda \bar u_h, u -\bar u_h) \geq 0, & \forall u \in U_{ad,h}.  \label{eq:var_ineqh}
\end{align}
\end{subequations}
\end{theorem}

Later on, in the derivation of the order of convergence for the optimal control, we shall need this optimality system as well as the following estimate for the adjoint equation.
\begin{theorem}\label{t:ph_estimate} Let $(y(u), u)$ a feasible pair  for problem $\mathbf{(P)} $, with $\norm{B u} <\nu^2$ and let the adjoint state $p(u)$ solution of the following equation:
\begin{equation}\label{eq:adj_p}
a(p(u),\varphi) + b(y(u),\varphi, p(u)) + b(\varphi,y(u),p(u)) = (y(u)-y_d,\varphi), \quad \forall \varphi \in H_0^1(0,1).
\end{equation}
If $p_h(u) \in V_h$ is the solution of the discretized version of equation \eqref{eq:adj_p}; that is:
\begin{equation}\label{eq:discrete_adjoint}
a(p_h,\varphi_h)+ b(y_h(u),\varphi_h, p) + b(\varphi_h,y_h(u),p) =( y_h ( u)-y_{d},\varphi_h) ,\quad   \forall \varphi_h \in V_h,
\end{equation}
 then there exists a constant $C$, independent of $h$ such that
\begin{equation} \label{eq:ph_H1}
\norm{ p(u) -p_h(u)}_{H_0^1(0,1)} \leq Ch,
\end{equation}
moreover,  the  estimate in the $L^2$-norm holds:
\begin{equation} \label{eq:ph_L2}
\norm{ p(u) - p_h(u)}\leq Ch^2.
\end{equation}
\end{theorem}
\begin{proof}
The proof is analogous to the proof of Theorem \ref{t:burg_est_l2}. To simplify notation we define $p=p(u)$ and $p_h=ph_(u)$. Let us take $\varphi_h \in V_h$ as test function in the weak formulation of \eqref{eq:adj_p} and  \eqref{eq:discrete_adjoint} respectively, and then substract the resulting equations obtaining
\begin{align}
a(p_h- p,\varphi_h)&=b(y(u),\varphi_h,  p )-b(y_h(u),\varphi_h, p_h ) \nonumber \\
&+b(\varphi_h, y(u), p )-b(\varphi_h,y_h(u), p_h ) \nonumber \\
&+ (y_h(u)-y(u),\varphi_h ) ,  \label{eq:est_p.0}
\end{align}
by applying property \eqref{eq:b} and choosing $\varphi_h = p_h-\Pi_h p$, from \eqref{eq:est_p.0} we have that
\begin{align}
a(p_h-&\Pi_h p,p_h-\Pi_h p) \nonumber \\
=&a( p-\Pi_h p,p_h-\Pi_h p)+b(y(u)-y_h(u),p_h-\Pi_h p,  p ) \nonumber \\
&+ b(y_h(u), p_h-\Pi_h p, p-p_h)+ b(p_h-\Pi_h p, y(u)-y_h(u), p ) \nonumber \\
&+  b(p_h-\Pi_h p, y_h(u), p-p_h ) + (y(u)-y_h(u), p_h-\Pi_h p)\nonumber \\
=&a( p-\Pi_h p,p_h-\Pi_h p)+b(y(u)-y_h(u),p_h-\Pi_h p,  p ) \nonumber \\
&+ b(y_h(u), p_h-\Pi_h p, p-\Pi_{h}p)+ b(p_h-\Pi_h p, y(u)-y_h(u), p ) \nonumber \\
&+  b(p_h-\Pi_h p, y_h(u), p-\Pi_{h}p ) + b(p_h-\Pi_h p, y_h(u), \Pi_{h}p-p_h ) \nonumber\\
& +  (y(u)-y_h(u), p_h-\Pi_h p). \label{eq:est_{p.0a}}
\end{align}
Taking into account the definition of the trilinear form $b$, the property \eqref{eq:b0} and Lemma \ref{l:local_sol} we estimate
\begin{align}
&\nu \norm{p_h-\Pi_h p}_{H_0^1(0,1)}  \nonumber \\
& \leq (2\norm{y_h(u)}_{H_0^1(0,1)} + \nu)\norm{p-\Pi_h p}_{H_0^1(0,1)} +  \norm{y_h(u)}_{H_0^1(0,1)} \norm{p_h-\Pi_{h}p}_{H_0^1(0,1)}  \nonumber\\
& \quad + (2\norm{p}_{H_0^1(0,1)}+1)\norm{y(u)-y_h(u)}_{H_0^1(0,1)} \nonumber \\
&  \leq (\frac2\nu\norm{B u} + \nu)\norm{p-\Pi_h p}_{H_0^1(0,1)} + (2\norm{p}_{H_0^1(0,1)}+1)\norm{y(u)-y_h(u)}_{H_0^1(0,1)} \nonumber \\
& \quad +  \frac1\nu\norm{Bu}_{H_0^1(0,1)} \norm{p_h-\Pi_{h}p}_{H_0^1(0,1)} \label{eq:est_p.0b}
\end{align}
Taking the last term in \eqref{eq:est_p.0b} to the left-hand side, since $\norm{Bu} \leq \nu^2$ and using estimates \eqref{eq:burg_est_h1} and Lemma \ref{l:proj_err} we have that there exist a constant $C>0$ independent of $h$ such that
\begin{align}
\norm{p_h-\Pi_h p}_{H_0^1(0,1)}\leq & C (\norm{p-\Pi_h  p}_{H_0^1(0,1)} + \norm{y(u)-y_h(u)}_{H_0^1(0,1)} )\nonumber \\
 \leq & Ch ( \norm{ p}_{H^2(0,1)}+\norm{ y(u)}_{H^2(0,1)}). \nonumber
\end{align}
From the last inequality we obtain the desired $H_0^1$-estimate \eqref{eq:ph_H1} since
\begin{align}
\norm{p-p_h}_{H_0^1(0,1)} \leq& \norm{p-\Pi_h p}_{H_0^1(0,1)} + \norm{\Pi_h p-p_h}_{H_0^1(0,1)} \leq  C h. \nonumber
\end{align}
for some constant $C>0$ independent of $h$.
%In order to prove \eqref{eq:ph_L2}
Now, we prove  \eqref{eq:ph_L2}. By similar arguments used to derive \eqref{eq:burg_est_L2} we consider the auxiliary  linear elliptic problem \eqref{eq:aux_1}, with $y=y(u)$:
\begin{equation} \label{eq:aux_p}
\left\{
\begin{array}{ll}
&\text{Given} \, r\in L^2(0,1), \,\text{find} \, z \in H_0^1(0,1) \, \text{such that:} \\
&a(z,\varphi)+b( y(u),z,\varphi)+b(z,  y(u), \varphi)=(r,\varphi), \quad \forall \varphi \in  H_0^1(0,1), 
\end{array}
\right.
\end{equation} 
Once again, the bilinear form $\tilde a(z,\varphi):=a(z,\varphi)+b( y(u),z,\varphi)+b(z,  y(u), \varphi)$ is elliptic. Therefore, \eqref{eq:aux_p} has a unique solution $z \in H_{0}^{1}(0,1)$ with $\norm{z}_{H_{0}^{1}(0,1)}\leq \norm{r}$.  We denote by $z_h \in V_h$ the corresponding  finite element approximation \eqref{eq:aux_2}, which fulfills the error estimate $\norm{z-z_h} _{H_0^1(0,1)} \leq C h$, for some constant $C>0$ independent of $h$. After subtracting the auxiliary problem \eqref{eq:aux_p} and its discretization choosing $\varphi=p_h-p$ we get
\begin{align}
(r,p_h-p )= & a(z, p - p_h) + b( y(u),z, p_h-p)+ b(z, y(u), p_h-p) \nonumber \\
				     =&a(z-z_h, p - p_h)+a(z_h, p - p_h)  + b(y(u),z, p_h-p)+ b(z, y(u), p_h-p). \label{eq:est_p.3}
\end{align}
Choosing $\varphi_h=z_h$ in identity \eqref{eq:est_p.0} and inserting in \eqref{eq:est_p.3} leads to
\begin{align}
(r,p_h-p )= & a(z-z_h, p - p_h)+(y_h(u)- y(u), z_{h}) \nonumber \\
&+ b( y(u) , z_h, p )-b(y_h(u),z_h, p_h ) +b(z_h,y(u), p )-b(z_h,y_h(u), p_h ) \nonumber \\
&+ b( y(u) , z, p_h-p)+ b(z,y(u), p_h-p) \nonumber \\
= & a(z-z_h, p - p_h)+(y_h(u)-y(u), z_{h}) \nonumber \\
&+ b(y(u),z-z_h, p_h-p )+b(z-z_h,y(u), p_h-p ) \nonumber \\
&-b(y_h(u)-y(u),z_h, p_h ) -b(z_h,y_h(u)-y(u), p_h ), \nonumber \\
= & a(z-z_h, p - p_h)+(y_h(u)-y(u), z_{h}) \nonumber \\
&+ b(y(u),z-z_h, p_h-p )+b(z-z_h,y(u), p_h-p )\nonumber \\
&+ (z_{h}p_{h}', y_{h}(u)-y(u)), \nonumber 			  
\end{align}
since $a$ is continuous in $(H_0^1(0,1))^{2}$ and $b$ satisfies \eqref{eq:b0}, from estimates \eqref{eq:burg_est_h1} and \eqref{eq:burg_est_L2}similarly to Theorem  \ref{t:burg_est_l2} we get the estimate \eqref{eq:ph_L2}.
\end{proof}

We are interested in the convergence properties of local solutions of $\mathbf{(P_h)} $.  Following  ideas given in \cite{casmat01}, the following  convergence properties are established.
\begin{assumption}\label{eq:A1} In order to establish an order of convergence, we will assume the following 
\begin{equation}
\delta_\nu:=\ds\sup_{h>0} \{ \nu^2 - \norm{B u_h} \}>0.
\end{equation}
\end{assumption}
This uniform bound for $\bar u_h$ is needed to have the following result.
\begin{corollary}
We notice that in our notation $\bar p_h = p_h(\bar u_h) $ and $\bar p = p (\bar u)$; consequently,  under the assumption that   $\delta_\nu:=\ds\sup_{h>0} \{ \nu^2 - \norm{B u_h} \}>0$ then Theorem \ref{t:ph_estimate} implies that
\begin{align} \label{eq:adj_order}
\norm{\bar p_h - \bar p} &\leq \norm{p_h (\bar u_h)-p(\bar u_h)} + \norm{p (\bar u_h)-\bar p (\bar u)} \nonumber \\
&\leq  c( h^2 + \norm{\bar u - \bar u_h}) ,
\end{align}
for some constant $c$ independent of the size of the mesh $h$.
\end{corollary}
\begin{theorem} \label{t:conv1}
Let $(\bar y_h, \bar u_h)$ an optimal pair for $\mathbf{(P_h)} $. Then, there is a  subsequence $(\bar u_h)_{h>0}$ which converges weakly in $L^2(0,1)$ to a limit $\bar u$. Moreover, the weak limit $\bar u$ is a solution of $\mathbf{(P)} $ which satisfies
\begin{equation}\label{eq:convergence}
\lim_{h\arrow 0} J(\bar y_h,\bar u_h)=J(\bar y,\bar u)=\inf(P) \quad \text{and} \quad \lim_{h\arrow 0} \norm{\bar u_h -\bar u}=0.
\end{equation}
\end{theorem}
\begin{proof}
Let  $(\bar u_h)_{h>0}$  be the sequence such that $\bar u_h$ is the optimal control for $(\mathbf{P_h})$. Since this sequence is formed by admissible controls for the problem $\mathbf{(P_h)} $ then it is also bounded. Thus, we can extract a weak convergent subsequence in $L^2(0,1)$, denoted again by $(\bar u_h)_{h>0}$. Let us denote its weak limit by $\bar u$ and denote by $\bar y :=y(\bar u )$ the associated state. It is clear, in view of  Lemma \ref{l:st_strngconv}, that $\bar y_h \arrow  \bar y$ in $H_0^1(0,1)$. Noticing that the pair $(y_h(\Pi_h \bar u), \Pi_h \bar u)$ is feasible for $(\mathbf{P_h})$  and by convexity of the objective functional we can conclude that
\begin{align}
J(\bar y, \bar u) \leq & \liminf_{h\arrow 0} J(\bar y_h, \bar u_h) \nonumber \\
		 \leq &  \liminf_{h\arrow 0} J(y_h(\Pi_h \bar u), \Pi_h \bar u) \nonumber \\
		\leq &  \limsup_{h\arrow 0} J(y_h(\Pi_h \bar u), \Pi_h \bar u) = J(\bar y, \bar u) =\inf(\mathbf P),\nonumber 
\end{align}
which  together with the fact that $(\bar y, \bar u) \in \mathcal{F}_{ad}$ imply that $\bar u$ is an optimal control for $(\mathbf P)$. Then, the first identity of \eqref{eq:convergence} follows by applying Mazur's Theorem and convexity of the objective functional.
The second identity of  \eqref{eq:convergence} is obtained by the following argument
\begin{align}
\frac{\lambda}{2}\norm{\bar u_h - \bar u}^2&=\frac{\lambda}{2}\norm{\bar u_h}^2-\frac{\lambda}{2}\norm{\bar u}^2+\lambda (\bar u, \bar u_h -\bar u) \nonumber \\
&=J(\bar y_h, \bar u_h)-J(\bar y, \bar u)+\frac12\norm{\bar y - \bar y_h}^2+\lambda (\bar u, \bar u_h -\bar u),  
\end{align}
where the last term tends to 0 as $h \arrow 0$ by the weak convergence of $\bar u_h$ to $\bar u$ and the first part of  \eqref{eq:convergence}.
\end{proof}
\subsection{Derivation of the order of convergence}\label{s:errorest}
In this section we derive the main result of this paper by adapting the theory developed in \cite{cas06}. First we recall some auxiliary results. 
We denote the solution of the adjoint equation by $p(\bar u_{h})$ satisfying
\begin{equation}\label{eq:adj_uh}
a(p, \varphi) + b(\bar y_{h}, \varphi, p) + b(\varphi,\bar y_{h}, p) = ( \bar y_{h} - y_{d}, \varphi),  
\end{equation}
for all $\varphi \in H_0^1(0,1)$, where $\bar y_{h}$ is the solution of \eqref{eq:state_h}.
\begin{lemma}\label{l:left_est}
Let us assume that there is a constant $c>0$ independent of $h$ such that $\norm{\bar u_h -\bar u} \geq ch$, then there is an $h_0>0$ and $\mu>0$ such that the estimate
\begin{align}\label{eq:left_est}
\mu \norm{\bar u_{h}-\bar u }^2 \leq
&[\mathcal{L'}(y(\bar u_{h}),\bar u_{h}, p(\bar u_{h}))-\mathcal{L'}(\bar y,\bar u, \bar p)](y(\bar u_{h})-\bar y,\bar u_{h}  - \bar u),
\end{align}
is satisfied for all $h<h_0$.
\end{lemma}
\begin{proof}
From the first derivative of the Lagrangian given by \eqref{eq:d1_Lagr}, it follows that
\begin{align}
[\mathcal{L'}(y(\bar u_h),&\bar u_h, p(\bar u_h)) -\mathcal{L'}(\bar y,\bar u, \bar p)](y(\bar u_h)-\bar y,\bar u_h - \bar u) \nonumber \\
=&[\mathcal{L'}(y(\bar u_h),\bar u_h, p(\bar u_h)) -\mathcal{L'}(y(\bar u_h),\bar u_h,\bar p)](y(\bar u_h)-\bar y,\bar u_h - \bar u)\nonumber \\
&+[\mathcal{L'}(y(\bar u_h),\bar u_h, \bar p) -\mathcal{L'}(\bar y,\bar u, \bar p)](y(\bar u_h)-\bar y,\bar u_h - \bar u)  \nonumber \\
=& [\mathcal{L'}(y(\bar u_h),\bar u_h, p(\bar u_h)) -\mathcal{L'}(y(\bar u_h),\bar u_h, \bar p)](y(\bar u_h)-\bar y,\bar u_h - \bar u) \nonumber \\
&+\norm{y(\bar u_h) -\bar y}^2+\lambda \norm{\bar u_h -\bar u}^2  -2b(y(\bar u_h)-\bar y, y(\bar u_h)-\bar y, \bar p).
\label{eq:error.1}
\end{align}
Since $y(\bar u_h)$ and $\bar y =y(\bar u)$ satisfy \eqref{eq:burg1} for $f=\bar u_h$ and $f=\bar u$ respectively, the first term on the right-hand side in \eqref{eq:error.1} satisfies:
\begin{align}
[\mathcal{L'}(y&(\bar u_h), \bar u_h, p(\bar u_h)) -\mathcal{L'}(y(\bar u_h),\bar u_h,\bar p)](y(\bar u_h)-\bar y,\bar u_h - \bar u)  \nonumber \\
=&-\langle R'(y(\bar u_h), \bar u_h)(y(\bar u_h)- \bar y, \bar u_h - \bar u),  p(\bar u_h)-\bar p) \nonumber \\
=&-a(y(\bar u_h)-\bar y,p(\bar u_h) - \bar p ) -b(y(\bar u_h) -\bar y, y(\bar u_h),p(\bar u_h) - \bar p) - b(y(\bar u_h), y(\bar u_h) -\bar y,p(\bar u_h) - \bar p) \nonumber \\
&+(B(\bar u_h-\bar u), p(\bar u_h)- \bar p) \nonumber \\
=& b(y(\bar u_h), y(\bar u_h), p(\bar u_h) - \bar p) - (B\bar u_h, p(\bar u_h) - \bar p) - b(\bar y, \bar y, \bar p_h - \bar p) + (B\bar u,  p(\bar u_h) - \bar p) \nonumber \\
&-b(y(\bar u_h) -\bar y, y(\bar u_h), p(\bar u_h) - \bar p) - b(y(\bar u_h), y(\bar u_h) -\bar y, p(\bar u_h)- \bar p)+(B(\bar u_h-\bar u), p(\bar u_h) - \bar p) \nonumber \\
=&-b(y(\bar u_h) - \bar y, y(\bar u_h) -\bar y, p(\bar u_h) - \bar p) \label{eq:error.0},
\end{align}
by replacing  \eqref{eq:error.0} in  \eqref{eq:error.1}   we observe
\begin{align}
[\mathcal{L'}&(y(\bar u_{h}),\bar u_{h}, p(\bar u_{h}))-\mathcal{L'}(\bar y,\bar u, \bar p)](y(\bar u_{h})-\bar y,\bar u_{h} - \bar u) \nonumber \\
=&\norm{y(\bar u_{h})-\bar y}^2+\lambda \norm{\bar u_{h} -\bar u}^2-2b(y(\bar u_{h})-\bar y, y(\bar u_{h})-\bar y, \bar p) \nonumber \\
&- b(y(\bar u_{h}) - \bar y, y(\bar u_{h}) -\bar y, p(\bar u_{h}) - \bar p).
\label{eq:error.1a}
\end{align}
Now, let us define the sequences $z_{h}=\ds\frac{y(\bar u_{h})-\bar y}{\norm{\bar u_{h}-\bar u}}$ and  $v_{h}=\ds\frac{\bar u_{h}-\bar u}{\norm{\bar u_{h}-\bar u}}$. By noticing that these sequences are bounded, we can extract a subsequence of $h$ denoted again by $h$, such that $h \arrow 0$  and  $v_{h} \rightharpoonup v$ in $L^2(0,1)$ and $z_{h} \rightharpoonup z$  in $H_0^1(0,1)$. 
We shall proof that $(z,v)$ belongs to the critical cone $C_{\bar u}^\tau$.  By its definition, $v_h$ fulfills \eqref{eq:crit_cone_b} and \eqref{eq:crit_cone_c}, and so does $v$. In order to check that $v(x)=0$ whenever $x \in \om_{\tau}$, we argue as in \cite{cas06}. If $h \arrow 0$, from Theorem \ref{t:conv1}, we have that there also exists a subsequence $\bar u_{h}(x) \arrow \bar u (x)$ a.e. $x \in (0,1)$ 
We notice that since $\bar{u}_{h} \arrow \bar u$  the associated adjoint state $\bar p_{h} \arrow \bar p$, both in in $L^2(0,1)$ which together with  $v_{h} \rightharpoonup v$ imply 
\begin{align}
(B\bar p +\lambda \bar u, v) \nonumber & =\lim_{h \arrow 0} (B\bar p_{h} + \lambda \bar u_{h}, v_{h}) \nonumber \\
						&=\lim_{h \arrow 0} \frac{1}{\norm{\bar u_{h}-\bar u}} (B\bar p_{h} + \lambda \bar u_{h}, \bar u_{h}-\bar u)  \nonumber \\
						& = \lim_{h \arrow 0} \frac{1}{\norm{\bar u_{h}-\bar u}} (B\bar p_{h} + \lambda \bar u_{h}, \bar u_h -\Pi_{h} \bar u) + \frac{1}{\norm{\bar u_{h}-\bar u}} (B\bar p_{h} + \lambda \bar u_{h}, \Pi_{h} \bar u - \bar u).\nonumber 
\end{align}						
By considering \eqref{eq:var_ineqh} we infer
\begin{align}						
(B\bar p +\lambda \bar u, v) 						& \leq \lim_{h \arrow 0} \frac{1}{\norm{\bar u_{h}-\bar u}} (B(\bar p_{h}-\bar p) + \lambda (\bar u_{h}-\bar u), \Pi_{h} \bar u - \bar u)  \nonumber \\
						&\qquad +\lim_{h \arrow 0} \frac{1}{\norm{\bar u_{h}-\bar u}}(B\bar p +\lambda \bar u, \Pi_{h} \bar u - \bar u)\nonumber \\
						&    \leq \lim_{k \arrow 0} \frac{\norm{\Pi_{h} \bar u - \bar u }}{\norm{\bar u_{h}-\bar u}} \l(  \norm{B(\bar p_{h} -\bar p)} +\lambda \norm{\bar u_{h} - \bar u}  \r)+ \norm{B\bar p+ \lambda \bar u} \lim_{h \arrow 0} \frac{\norm{\Pi_{h} \bar u - \bar u}}{\norm{\bar u_{h}-\bar u}} \nonumber \\
						& \leq c\lim_{h \arrow 0}  \frac{\norm{\Pi_{h} \bar u - \bar u }}{h}=0, \nonumber
\end{align}
where the last estimation follows from Lemma \ref{l:proj_err}. From the last inequality and the fact that $\tau <|\bar p(x) + \lambda \bar u(x)|$, we infer that $v(x)=0$ in $\om_{\tau}$ and thus we have that $v \in C_{\bar u}^\tau$. Furthermore, the pair $(z,v)$ satisfies the linear equation given by $R'(\bar y, \bar u)(z,v)=0$. To see this,  substract $\langle R(y(\bar u_{h}), \varphi \rangle$ and  $\langle R(\bar y, \bar u), \varphi \rangle$ to obtain
\begin{equation}  \label{eq:lineq_wk2}
a(z_{h},\varphi)+b(z_{h}, y(\bar u_{h}),\varphi)+b(\bar y, z_{h},\varphi)-(B v_{h},\varphi)=0 \quad \forall \varphi \in H_0^1(0,1),
\end{equation}
hence,  using the convergence properties of $z_{h}$ and $v_{h}$ and  Lemma \ref{l:st_strngconv},  after passing to the limit $k \arrow \infty$, we get $R'(\bar y, \bar u)(z,v)=0$.

From \eqref{eq:error.1a}   we have
\begin{align}
[\mathcal{L'}&(y(\bar u_{h}),\bar u_{h}, p(\bar u_{h}))-\mathcal{L'}(\bar y,\bar u, \bar p)](y(\bar u_{h})-\bar y,\bar u_{h} - \bar u) \nonumber \\
=&\l(\norm{z_{h}}^2+\lambda \norm{v_{h}}^2-2b(z_{h}, z_{h}, \bar p)  -b(z_{h}, z_{h}, \bar p(\bar u_{h}) - \bar p )  \r) \norm{\bar u_{h}-\bar u }^2,
\label{eq:error.1ab}
\end{align}

Moreover,  since the pair $(z,v) \in C^\tau_{\bar u}$ with $R'(\bar y, \bar u)(z,v)=0$ and $p(\bar u_{h}) \arrow \bar p$ uniformly, by applying second order sufficient conditions \eqref{eq:ssc} we infer that
\begin{align}
\liminf_{h \arrow 0} & \l( \norm{z_{h}}^2+\lambda \norm{v_{h}}^2-2b(z_{h}, z_{h}, \bar p) -b(z_{h}, z_{h}, p(\bar u_{h})- \bar p)   \r) \nonumber \\
& \geq \norm{z}^2+\lambda \norm{v}^2-2b(z, z,\bar p) \geq \delta \norm{v}^2. \label{eq:error.2c}
\end{align}
Finally, from \eqref{eq:error.1ab}, \eqref{eq:error.2c} and choosing $\mu=\frac{\delta}{2}\norm{v}^2$,  there exists an $h_0$ such that for all $h< h_0$ we get
\begin{align}
\mu \norm{\bar u_{h}-\bar u }^2 \leq
&[\mathcal{L'}(y(\bar u_{h}),\bar u_{h}, p(\bar u_{h}))-\mathcal{L'}(\bar y,\bar u, \bar p)](y(\bar u_{h})-\bar y,\bar u_{h}  - \bar u).\nonumber
\end{align}
\end{proof}
\begin{lemma}\label{l:gradient.1} It holds that
\begin{equation}
\lim_{h \arrow 0} \frac{1}{h^2} |\bigl( (B \bar p +\lambda \bar u) , \Pi_h \bar u-\bar u \bigl)| =0.
\end{equation}
\end{lemma}
\begin{proof}
The proof of this Lemma can be found in \cite[Lemma 4.4]{cas06}.
\end{proof}
\begin{theorem} \label{t:mainresult} Let $\bar u$ be a local solution of  $\mathbf{(P)}$ satisfying second-order sufficient condition \eqref{eq:ssc}. If $\bar u_h$ is a local solution of $\mathbf{(P_h)} $ satisfying Assumption \ref{eq:A1} and  such that $\lim_{h\arrow 0}\norm{\bar u_h-\bar u}=0$ then the following convergence property holds
\begin{equation}\label{eq:estimate}
 \lim_{h\arrow 0}\frac{1}{h}\norm{\bar u_h-\bar u}=0.
\end{equation}
\end{theorem}
\begin{proof}
As in \cite{cas06}, the proof is argued by contradiction. Let us assume that \eqref{eq:estimate} is false, therefore there exist a constant $c>0$  and a subsequence $(\bar u_{h})_{h>0}$ such that the relation 
\begin{equation} \label{eq:noestimate}
\norm{\bar u_{h} - \bar u} \geq c h,
\end{equation}
holds for all $h>0$ sufficiently small. From Lemma \ref{l:left_est} there exists $h_0>0$ such that for all $h<h_0$ the estimate \eqref{eq:left_est} holds.

We proceed to estimate the right hand side of \eqref{eq:left_est}. Since $\bar u$ and $\bar u_h$ are local solutions for  $\mathbf{(P)}$ and $\mathbf{(P_h)}$  respectively, then they satisfy the first order necessary conditions given by Theorems \ref{t:fonc} and \ref{t:fonch}.  We observe that $\bar u_h $ is feasible  for $\mathbf{(P)}$ and $\Pi_h \bar u$ is feasible for $\mathbf{(P)}$. Therefore, taking  $u=\bar u_h$ in \eqref{eq:var_ineq} and   $u=\bar u$ in \eqref{eq:var_ineqh} it comes
\begin{align}
&(B\bar p +\lambda \bar u, \bar u_h - \bar u) \geq 0 , \text{and} \label{eq:vi1} \\
&(B\bar p_h +\lambda \bar u_h, \Pi_h\bar u - \bar u_h) \geq 0 \label{eq:vi2},
\end{align}
Then, inequalities \eqref{eq:vi1}  and \eqref{eq:vi2} imply that
\begin{align}
\mathcal{L'}&(\bar y_h,\bar u_h, \bar p_h)(\bar y_h-\Pi_h \bar y, \bar u_h - \Pi_h  \bar u)  - \mathcal{L'}(\bar y,\bar u, \bar p)(\bar y_h-\bar y, \bar u_h - \bar u)& \nonumber  \\
=& (\bar p_h, B(\bar u_h -\Pi_h \bar u))+ \lambda (\bar u_h, \bar u_h - \Pi \bar u) - (\bar p, B(\bar u_h -\bar u))- \lambda (\bar u, \bar u_h - \bar u) \nonumber \\
=& (B \bar p_{h} + \lambda \bar{u}_{h}, \bar u_{h}-\Pi_{h}\bar u  ) - (B \bar p + \lambda u, \bar u_{h}-u) \leq 0 \label{eq:erro.0a} .
 \end{align}
 With the help of \eqref{eq:erro.0a} we estimate
\begin{align}
[\mathcal{L'}&(y(\bar u_{h}),\bar u_{h}, p(\bar u_{h}))-\mathcal{L'}(\bar y,\bar u, \bar p)](y(\bar u_{h})-\bar y,\bar u_{h} - \bar u) \nonumber\\
=&\mathcal{L'}(y(\bar u_h),\bar u_h,  p(\bar u_h))(y(\bar u_h)-\bar y,\bar u_h - \bar u ) -\mathcal{L'}(\bar y_h,\bar u_h, \bar p_h)( \bar y_h -\Pi_h \bar y, \bar u_h -\Pi_h \bar u) \nonumber \\
&+\mathcal{L'}(\bar y_h,\bar u_h, \bar p_h)( \bar y_h -\Pi_h \bar y, \bar u_h -\Pi_h \bar u)-\mathcal{L'}(\bar y,\bar u, \bar p)(y(\bar u_h)-\bar y,\bar u_h - \bar u) \nonumber \\
\leq& (y(\bar u_h) - y_d,y(\bar u_h)-\bar y ) + \lambda (\bar u_h,\bar u_h -\bar u)  -  \langle R'(\bar y(\bar u_h), \bar u_h) (y(\bar u_h)- \bar y, \bar u_h - \bar u ), p(\bar u_h) \rangle\nonumber \\
&-( \bar y_h -y_d, \bar y_h- \Pi_h y) - \lambda ( \bar u_h, \bar u_h-\Pi_h \bar u)  + \langle R'(\bar y_h, \bar u_h) (\bar y_h-\Pi_h \bar y, \bar u_h- \Pi_h \bar u), \bar p_h \rangle\nonumber 
\end{align}

since $\bar p_{h}$ and  $p (\bar u_{h})$ satisfy adjoint equations  \eqref{eq:adj_h} and \eqref{eq:adj_uh} we have
\begin{align}
=& (B p(\bar u_{h}), \bar u_{h} -\bar u)+\lambda (\bar u_{h}, \bar u_{h} -\bar u) - ( B \bar p_{h}, \bar u_{h}-\Pi_{h} \bar u) - \lambda ( \bar u_{h}, \bar u_{h}- \Pi_{h}\bar u) \nonumber\\
=&   (\lambda \bar u_{h}+B \bar p_{h},  \Pi_{h} \bar u- \bar u)+ ( B(p(\bar u_{h})- \bar p_{h}), \bar u_{h}-\bar u) \nonumber \\
=& (\lambda  (\bar u_{h}-\bar u)+B(\bar p_{h} -\bar p),  \Pi_{h} \bar u- \bar u)+ (\lambda \bar u + B \bar p,\Pi_{h} \bar u- \bar u )+ ( B(p(\bar u_{h})- \bar p_{h}), \bar u_{h}-\bar u) \label{eq:error.0b}
\end{align}
Now, from our assumption \eqref{eq:noestimate} we apply Lemma \ref{l:left_est} to \eqref{eq:error.0b} we get
\begin{align}
\mu \norm{\bar u_h -\bar u}^2 &\leq 
[\mathcal{L'}(y(\bar u_{h}),\bar u_{h}, p(\bar u_{h}))-\mathcal{L'}(\bar y,\bar u, \bar p)](y(\bar u_{h})-\bar y,\bar u_{h} - \bar u) \nonumber\\
&\leq 
 (\lambda  (\bar u_{h}-\bar u)+B(\bar p_{h} -\bar p),  \Pi_{h} \bar u- \bar u)+ (\lambda \bar u + B \bar p,\Pi_{h} \bar u- \bar u )\nonumber\\
 &\quad+ ( B(p(\bar u_{h})- \bar p_{h}), \bar u_{h}-\bar u) \label{eq:error.0b1}
 \end{align}
 According to the estimate \eqref{eq:adj_order}, from \eqref{eq:error.0b1} we have
 \begin{align}
\mu \norm{\bar u_h -\bar u}^2 &  
\leq (\lambda \norm{\bar u_h -\bar u} + \norm{\bar p_{h} -\bar p} )\norm{\Pi_{h} \bar u- \bar u } + (\lambda \bar u + B \bar p,\Pi_{h} \bar u- \bar u ) \nonumber\\
&\quad+\norm{ p_{h}(\bar u_h) -\bar p_h}\norm{\bar u_h -\bar u} \nonumber\\
& \leq c_{1} (\norm{\bar u -\bar u_h}+h^2) \norm{\Pi_{h} \bar u- \bar u } +  (\lambda \bar u + B \bar p,\Pi_{h} \bar u- \bar u )+c_{2}h^{2}\norm{ \bar u_{h}-\bar u} \label{eq:error.0b2}
\end{align}
next, dividing the last relation by $ h\norm{\bar u_{h}-\bar u}$, then \eqref{eq:noestimate} implies
\begin{align}
\mu \norm{\bar u_{h}-\bar u} \leq &c \frac{\norm{  \Pi_{h} \bar u- \bar u }}{h}
+\frac{(\lambda \bar u + B \bar p,\Pi_{h} \bar u- \bar u ) }{h^{2}}
+ c_{2}h,
\end{align}
for some constant $c>0$ independent of $h$. Finally,  in view of Lemma \ref{l:gradient.1} and the relation \eqref{eq:superlinear1} we divide by $h$ and pass to the limit
\begin{align}
\mu \lim_{k \arrow \infty} 
\frac{1}{h}\norm{\bar u_{h}-\bar u}
 \leq &  \lim_{k \arrow \infty} \biggl( c \frac{\norm{  \Pi_{h} \bar u- \bar u }}{h}
+\frac{(\lambda \bar u + B \bar p,\Pi_{h} \bar u- \bar u ) }{h^{2}}\biggl)=0. \nonumber 
\end{align}
This contradicts our assumption \eqref{eq:noestimate}, and therefore the statement of Theorem \ref{t:mainresult} is true.
\end{proof}
\subsection{An improved error estimate} We make a further error analysis by taking into account a stronger assumption on the estructure of the optimal control which allow us to derive a better interpolation error in the $L^1$--norm which is crucial to make an improvement in the overall error estimate. The following assumption was proposed by R\"osch in \cite{roe06}, and guarantees that $\bar u$ is Lipschitz continuous and piecewise of class $C^2$ on the domain $\om=(0,1)$. 
\begin{assumption} \label{eq:A2} There exists a finite number of points $t_k \in[0,1]$, for $k=0,\ldots, N$ such that $t_0=0$ and $t_N=1$, such that the optimal control $\bar u \in C^2[t_{k-1},t_{k}]$ for all $k=1,\ldots,N$.
\end{assumption}
The following interpolation error is a consequence of the  last assumption and its proved in \cite[ Lemma 3]{roe06}.
\begin{lemma}\label{l:improved.0}
Under Assumption \ref{eq:A2} there exists a positive constant $c$, such that the following bound for the interpolation error
\begin{align}\label{eq:improved.0}
\norm{\bar u -\Pi_{h}\bar u} \leq c h^{3/2},
\end{align}
holds.
\end{lemma}
\begin{lemma}\label{l:improved.1}
Under Assumption \ref{eq:A2} there exists a positive constant $c$, such that the following estimate follows:
\begin{align}\label{eq:improved.1}
|(\lambda \bar u + B \bar p, \Pi_h \bar u -\bar u)| \leq ch^3.
\end{align}
\end{lemma}
\begin{proof} The proof is done along the lines of the proof of Lemma 4.4 in \cite{cas06}, in combination with the arguments of Lemma 3 in \cite{roe06}. In our mesh $\mathcal{I}_h$, we consider the following sets 
\begin{align}
I_h^+ &= \{ i:  |\lambda \bar u (x) + B \bar p(x)|>0, \, \forall x \in I_i \in \mathcal{I}_h\}, \,\text{and} \nonumber \\
I_h^0 &= \{ i:  \exists\, \xi_i \in I_i, \text{such that }  \lambda \bar u (\xi_i) + B \bar p(\xi_i)=0\}. \nonumber
\end{align}
Notice that $\lambda \bar u + B \bar p \in C^{0,1}$ by the regularity of $\bar u$ and $\bar p$. Moreover, from the variational inequality \eqref{eq:var_ineq} we have the characterization
\begin{align}
\lambda \bar u(x) + B \bar p(x) \geq 0 & \,\text{if }  \bar u (x) = \alpha, \nonumber \\
\lambda \bar u(x) + B \bar p(x)  \leq 0 & \,\text{if } \bar u (x) = \beta,  \nonumber\\
\lambda \bar u(x) + B \bar p(x)  = 0 & \,\text{if } \alpha <\bar u (x) <\beta .  \nonumber
\end{align}

 Therefore, if $i \in I_h^+$ we have that $\lambda \bar u(x) + B \bar p(x) \not =0$ for all $x \in I_i$, and thus $\bar u(x)=\alpha$ or $\bar u(x)=\beta$ accordingly, which in turn implies that $\Pi_h \bar u(x) = \bar u(x)$  for all $x \in I_i$. From this observation we have:
 \begin{align}\label{eq:improved.11}
 |(\lambda \bar u + B \bar p, \Pi_h \bar u -\bar u)| &= \left|\int_0^1 (\lambda \bar u(x) + B \bar p(x))( \Pi_h \bar u(x) -\bar u(x)) \,dx \right| \nonumber \\
 & = \left| \sum_{i =1}^n\int_{I_i} (\lambda \bar u(x) + B \bar p(x))( \Pi_h \bar u(x) -\bar u(x)) \,dx \right| \nonumber\\
 & = \left| \sum_{i \in I_h^0}\int_{I_i} (\lambda \bar u(x) + B \bar p(x))( \Pi_h \bar u(x) -\bar u(x)) \,dx \right| \nonumber\\
 & \leq \sum_{i \in I_h^0}\int_{I_i} |\lambda \bar u(x) + B \bar p(x) - \lambda \bar u(\xi_i) - B \bar p(\xi_i) | \, | \Pi_h \bar u(x) -\bar u(x) | \,dx. 
 \end{align}
Since $\lambda \bar u + B \bar p \in C^{0,1}$.  By denoting its Lipschitz constant by $\tilde L$; from \eqref{eq:improved.11} we have that
 \begin{align}
 |(\lambda \bar u + B \bar p, \Pi_h \bar u -\bar u)| & \leq \sum_{i \in I_h^0}\int_{I_i} L|x-\xi_i| \, | \Pi_h \bar u(x) -\bar u(x) | \,dx.  \nonumber\\
 & \leq Lh \sum_{i \in I_h^0} \int_{I_i}| \Pi_h \bar u(x) -\bar u(x) | \, dx.  \nonumber\\
 & \leq Lh  \int_{0}^1| \Pi_h \bar u(x) -\bar u(x) | \, dx.  \label{eq:improved.2}
 \end{align}
Now, consider the integrand on the right--hand side of \eqref{eq:improved.2}. By  Assumption \ref{eq:A2}, we distinguish the intervals $[t_{i-1}, t_{i}]$ between the class $\mathcal{I}_1$ containing the intervals where $\bar u \in C^2[t_{i-1}, t_{i}]$ and the class  $\mathcal{I}_2$ formed by the remaining intervals where $\bar u$ is only Lipschitz. From interpolation error for piecewise linear functions in one dimension, we get:
\begin{align}
\int_{0}^1| \Pi_h \bar u(x) -\bar u(x) | \, dx& = \sum_{i=1}^N \int_{t_{i-1}}^{t_{i}} | \Pi_h \bar u(x) -\bar u(x) |\, dx \nonumber \\
&= \sum_{I_1} \int_{t_{i-1}}^{t_{i}} | \Pi_h \bar u(x) -\bar u(x) |\, dx  + \sum_{I_2} \int_{t_{i-1}}^{t_{i}} | \Pi_h \bar u(x) -\bar u(x) |\, dx \nonumber \\
&\leq \sum_{I_1} \frac{\norm{\bar u''(\zeta_i)}}{8} h^2 \,h+ \sum_{I_2} \frac{L}{4}h\,h \, dx \nonumber\\
&\leq \sum_{I_1} \max_{1\leq i \leq N}\frac{\norm{\bar u''(\zeta_i)}}{8} h^3 \, h + (N-1)\frac{L}{2} h^2 \, dx  \label{eq:improved.3}
\end{align}
where $\zeta_i \in (t_{i-1}, t_{i})$. Note that by the Assumption \ref{eq:A2} the class $I_2$ contains at most $N-1$ intervals and that $N$ is independent of $h$. Therefore, from \eqref{eq:improved.2} and \eqref{eq:improved.3} we deduce the estimate \eqref{eq:improved.1}.
\end{proof}
\begin{theorem} \label{t:improved_result} Let $\bar u$ be a local solution of  $\mathbf{(P)}$ satisfying second-order sufficient condition \eqref{eq:ssc}. If $\bar u_h$ is a local solution of $\mathbf{(P_h)} $  such that $\lim_{h\arrow 0}\norm{\bar u_h-\bar u}=0$; under  Assumptions \ref{eq:A1} and \ref{eq:A2} then the following  error estimate holds
\begin{equation}\label{eq:estimate_improved}
\norm{\bar u_h-\bar u} \leq ch^{3/2}
\end{equation}
\end{theorem}
\begin{proof}
Analogous to the proof of Lemma \ref{l:left_est}, we have a pair of sequences  $z_{h}=\ds\frac{y(\bar u_{h})-\bar y}{\norm{\bar u_{h}-\bar u}}$ and  $v_{h}=\ds\frac{\bar u_{h}-\bar u}{\norm{\bar u_{h}-\bar u}}$ such that $v_{h} \rightharpoonup v$ in $L^2(0,1)$ and $z_{h} \rightharpoonup z$  in $H_0^1(0,1)$ when $h \arrow 0$. By construction of $v_h$, it is easy to see that $v$ satisfies the sign condition \eqref{eq:crit_cone_b} and \eqref{eq:crit_cone_c}. To verify that the condition \eqref{eq:crit_cone_a} is satisfied by $v$, we first observe that $\int_{\om_\tau} (B\bar p(x) + \lambda \bar u(x))( \bar u_h(x) -\bar u(x))\, dx \geq 0$, implying that  $\int_{\om_\tau} (B\bar p(x) + \lambda \bar u(x))v\, dx \geq 0$. 

On the other hand,
\begin{align}						
\int_{\om_\tau} (B\bar p(x) +& \lambda \bar u(x))v(x)\,dx \nonumber\\
&= \frac{1}{\norm{\bar u_h -\bar u}}\lim_{h \arrow 0} \int_{\om_\tau} (B\bar p_h(x) +\lambda \bar u_h(x))({ \bar u_h(x) -\bar u(x)})\,dx \nonumber \\
&=  \frac{1}{\norm{\bar u_h -\bar u}}\lim_{h \arrow 0} \l[\int_{\om_\tau} (B\bar p_h(x) +\lambda \bar u_h(x))( \bar u_h(x) -\Pi_h\bar u(x))\,dx \r. \nonumber \\
&\qquad\qquad\qquad \qquad \l.+  \int_{\om_\tau} (B\bar p_h(x) +\lambda \bar u_h(x))( \Pi_h\bar u(x)-\bar u(x))\,dx \r]. \nonumber
 \end{align}
The first integral on the right-hand side of the last identity is less or equal than 0 by the first order necessary optimality conditions \eqref{eq:var_ineqh}; while the second integral is equal to 0 by noticing that the optimal control is active on $\om_\tau$, i.e. $\bar u(x)=\alpha$ or $\bar u(x)=\beta$ and thus $\Pi_h \bar u(x) - \bar u(x) = 0$. Since $|B\bar p(x) + \lambda \bar u(x)|>\tau>0$ on $\om_\tau$, we apply second the order sufficient condition of Theorem \ref{t:ssc} and by the repeating the arguments in the proof of Lemma \ref{l:left_est}, we can deduce the existence of an $h_0>0$ and $\mu>0$, such that
\begin{align}\label{eq:left_est2}
\mu \norm{\bar u_{h}-\bar u }^2 \leq
&[\mathcal{L'}(y(\bar u_{h}),\bar u_{h}, p(\bar u_{h}))-\mathcal{L'}(\bar y,\bar u, \bar p)](y(\bar u_{h})-\bar y,\bar u_{h}  - \bar u), \nonumber \\
& \leq c_{1} (\norm{\bar u -\bar u_h}+h^2) \norm{\Pi_{h} \bar u- \bar u } +  (\lambda \bar u + B \bar p,\Pi_{h} \bar u- \bar u )+c_{2}h^{2}\norm{ \bar u_{h}-\bar u}. 
\end{align}
Now, by applying the Young's inequality and taking into account Lemma \ref{l:improved.0} and Lemma \ref{l:improved.1}, we finally deduce that
\begin{align}\label{eq:improved_restult}
\frac{\mu}{2} \norm{\bar u_{h}-\bar u }^2
& \leq c_2 h^4+c_3\norm{\Pi_{h}\bar u- \bar u }^2 +  (\lambda \bar u + B \bar p,\Pi_{h} \bar u- \bar u ) \nonumber \\
& \leq c_2 h^4 +c_4 h^3 + ch^3, 
\end{align}
which implies the error estimate \eqref{eq:estimate_improved}
\end{proof}

\section{Numerical experiments}
For the sake of illustration of our theory, we develop a numerical test where the exact solution of the optimal control problem is known. The optimization problem is solved by a BFGS method, which stops when the norm of  the residual $u_{k+1}-u_k$  is less than the tolerance of $1e-7$.  Our example reads as follows
\begin{subequations}\label{eq:example}
\begin{empheq}[left={\mathbf{(E)} \empheqlbrace}]{align}
\displaystyle &\min_{(y,u)\in H_0^1(0,1) \times U_{ad}} J({y,u})=\frac12\norm{ y-y_d}^2 +\frac\lambda2 \norm{u - u_d}^2\\
&\text{subject to: } \nonumber\\
&\begin{array}{l} \label{eq:burgers}
-\nu  y'' +yy' =u +f \quad \text{in } (0,1), \\
y(0)=y(1)=0,
\end{array}
\end{empheq}
\end{subequations}
where:
\begin{subequations}\label{eq:example.1}
\begin{align}
\nu &= 0.78, \nonumber\\
u_a &= -1, \quad u_b =1 \nonumber \\
y_d(x) &= -x (x -1), \nonumber\\
u_d(x) &=
\left\{
\begin{array}{ll}
1 & 0 \leq x <1/3, \\
- 6x +3 & 1/3 \leq x \leq 2/3, \\
-1 & 2/3 < x <1, 
\end{array}
\right.
\nonumber\\
f(x)&=-u_d + x(x-1)(2x-1)+2\nu.\nonumber
\end{align}
\end{subequations}
With these choices, problem $(E)$ has the optimal control $\bar u = u_d$ with associated optimal state $\bar y = y_d$, and adjoint state $\bar p =0$ together satisfying the optimality conditions stablished in Theorem \ref{t:fonc}. Note that since $\bar p$ vanishes, the optimal quantities also satisfy the second order optimality condition \eqref{eq:ssc}. Figure \ref{fig:opt_pair} shows the computed optimal control $\bar u$ and its associated optimal state $\bar y$ at $h=0.0244$.

In the next table we estimate numerically the order of convergence in the $L^2$--norm (EOC). From the numerical results, it can be observed a quadratic order of convergence for the optimal control values of $\lambda$ close to 1, but this order is lower if the value of $\lambda$ decreases. It should be notice that our control satisfies Assumption \ref{eq:A2}.
\begin{table}[h!]
\begin{tabular}{ccc}
$h$       & Error in $L^2$ & EOC \\
\hline
\hline
0.0476 & 0.0111045 &  - \\\hline
0.0244 & 0.0031017& 1.91 \\\hline
0.0123 & 0.0008099 & 1.94\\\hline
0.0062 & 0.0002152 & 1.95\\\hline
0.0031 & 0.0000656 & 1.94\\ \hline
0.0016 & 0.0000300 & 1.87\\ \hline
0.0008 & 0.0000190 & 1.71\\ \hline
\end{tabular}
\caption{Numeric computation of the errors and the order of convergence}
\end{table}

The numerical approximations can be observed in the following figure. Note that $\norm{\bar u} = \sqrt{7}/3 $; hence,  the condition $\norm{\bar u} \leq \nu^2$ is satisfied.
\begin{figure}[h]
        \centering
        \begin{subfigure}{0.5\textwidth}
                \centering
                \includegraphics[scale=0.35]{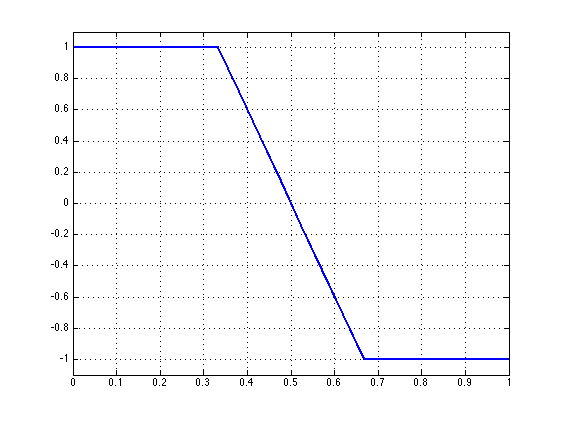}
                \caption{Optimal control}
                \label{fig:exp2_01}
        \end{subfigure}%
        ~ %add desired spacing between images, e. g. ~, \quad, \qquad etc.
        \begin{subfigure}[H]{0.5\textwidth}
                \centering
                \includegraphics[scale=0.35]{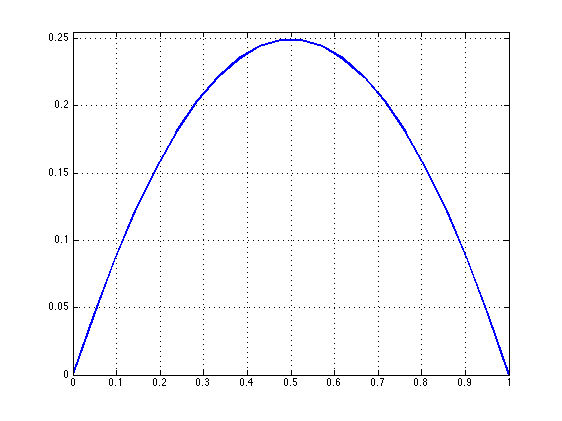}
                \caption{Optimal state}
                \label{fig:exp2_02}
        \end{subfigure}
 \caption{Approximation of the optimal pair $(\bar y,\bar u)$, for $\lambda=0.1$}     
  \label{fig:opt_pair}
\end{figure}

\bibliographystyle{plain}
%\bibliography{litbank.bib}

\end{document}